\documentclass[11pt]{amsart}

\usepackage{amscd,amssymb,amsopn,amsmath,amsthm,mathrsfs,graphics,amsfonts,enumerate,verbatim,calc
}

\usepackage[all]{xy}
\usepackage[titletoc,toc, title]{appendix}
\usepackage[normalem]{ulem}
\usepackage[colorlinks=true,linkcolor=blue,citecolor=blue]{hyperref}

\usepackage{color}

\usepackage[OT2,OT1]{fontenc}
\newcommand\cyr{%
\renewcommand\rmdefault{wncyr}%
\renewcommand\sfdefault{wncyss}%
\renewcommand\encodingdefault{OT2}%
\normalfont
\selectfont}
\DeclareTextFontCommand{\textcyr}{\cyr}

\DeclareFontFamily{OT1}{rsfs}{}
\DeclareFontShape{OT1}{rsfs}{n}{it}{<-> rsfs10}{}
\DeclareMathAlphabet{\mathscr}{OT1}{rsfs}{n}{it}

\numberwithin{equation}{section}
\newtheorem{theorem}{Theorem}[section]
\newtheorem{lemma}[theorem]{Lemma}
\newtheorem{proposition}[theorem]{Proposition}
\newtheorem{corollary}[theorem]{Corollary}
\newtheorem{claim}[theorem]{Claim}
\newtheorem{question}{Question}

\newtheorem*{maintheoremA}{Theorem A}
\newtheorem*{maintheoremB}{Theorem B}
\newtheorem{conjecture}[theorem]{Conjecture}

\theoremstyle{definition}
\newtheorem{definition}[theorem]{Definition}
\newtheorem{remark}[theorem]{Remark}
\theoremstyle{remark}

\newtheorem{example}[theorem]{Example}

\newcommand{\Ass}{\operatorname{Ass}}
\newcommand{\Assh}{\operatorname{Assh}}

\newcommand{\Spec}{\operatorname{Spec}}

\newcommand{\Ht}{\operatorname{ht}}

\newcommand{\Ext}{\operatorname{Ext}}

\newcommand{\Supp}{\operatorname{Supp}}

\newcommand{\Ann}{\operatorname{Ann}}

\newcommand{\Deg}{\operatorname{Deg}}
\newcommand{\hdeg}{\operatorname{hdeg}}
\newcommand{\udeg}{\operatorname{udeg}}

\newcommand{\Min}{\operatorname{Min}}
\newcommand{\rank}{\operatorname{rank}}

\newcommand{\ux}{\underline{x}}
\newcommand{\uy}{\underline{y}}

\newcommand{\fm}{\mathfrak{m}}

\newcommand{\fa}{\mathfrak{a}}

\newcommand{\fn}{\mathfrak{n}}

\newcommand{\length}{{l}}
\newcommand{\len}{l}
\newcommand{\mf}[1]{\mathfrak #1}

\renewcommand{\ge}{\geqslant} \renewcommand{\le}{\leqslant}
\renewcommand{\geq}{\geqslant} \renewcommand{\leq}{\leqslant}
 
\newcommand{\ol}{\overline}

\begin{document}
\title[Lech-Stuckrad-Vogel]{Lech's inequality, the St\"{u}ckrad--Vogel conjecture, and uniform behavior of Koszul homology}

\author[Patricia Klein]{Patricia Klein}
\address{Department of Mathematics, University of Kentucky, Lexington, KY 40506 USA}
\email{pklein@uky.edu}

\author[Linquan Ma]{Linquan Ma}
\address{Department of Mathematics, Purdue University, West Lafayette, IN 47907 USA}
\email{ma326@purdue.edu}

\author[Pham Hung Quy]{Pham Hung Quy}
\address{Department of Mathematics, FPT University, Hanoi, Vietnam}
\email{quyph@fe.edu.vn}

\author[Ilya Smirnov]{Ilya Smirnov}
\address{Department of Mathematics, Stockholm University, S-10691, Stockholm, Sweden}
\email{smirnov@math.su.se}

\author[Yongwei Yao]{Yongwei Yao}
\address{Department of Math and Statistics, Georgia State University, Atlanta, GA 30302 USA}
\email{yyao@gsu.edu}

\maketitle

\begin{abstract}
Let $(R,\fm)$ be a Noetherian local ring, and let $M$ be a finitely generated $R$-module of dimension $d$. We prove that the set $\left\{\frac{l(M/IM)}{e(I, M)} \right\}_{\sqrt{I}=\fm}$ is bounded below by ${1}/{d!e(\overline{R})}$ where $\overline{R}=R/\Ann(M)$. Moreover, when $\widehat{M}$ is equidimensional, this set is bounded above by a finite constant depending only on $M$. The lower bound extends a classical inequality of Lech, and the upper bound answers a question of St\"{u}ckrad--Vogel in the affirmative. As an application, we obtain results on uniform behavior of the lengths of Koszul homology modules.
\end{abstract}

\section{Introduction}

In \cite{Lech60}, Lech proved a simple inequality relating the Hilbert--Samuel multiplicity and the colength of an ideal. It states that if $(R,\fm)$ is a Noetherian local ring of dimension $d$ and $I$ is any $\fm$-primary ideal of $R$, then we have $$e(I, R)\leq d!e(R)l(R/I),$$ where $e(I, R)$ denotes the Hilbert--Samuel multiplicity of $I$ and $e(R)=e(\fm, R)$. In the same paper, Lech conjectured that, for every flat local extension $(R,\fm)\to (S,\fn)$ of Noetherian local rings, one has $e(R)\leq e(S)$. This conjecture is wide open in general.  Using the above inequality, Lech obtained the estimate $e(R)\leq d!e(S)$ where $d=\dim R$ \cite{Lech60}. We refer to \cite{HSV17} and \cite{Han02} for some generalizations of Lech's inequality and to \cite{Ma17} for recent progress on Lech's conjecture.

If we consider the set $\left\{ \frac{l(R/I)}{e(I, R)} \right\}_{\sqrt{I}=\fm}$ of positive numbers, then Lech's inequality is simply saying that this set is bounded below by $\frac{1}{d!e(R)}$ (and, thus, is bounded away from $0$). The infimum of this set was investigated by Mumford in his study of local stability \cite{Mum77}. In a different direction,
in \cite{SV96} St\"{u}ckrad and Vogel studied whether $\left\{ \frac{l(R/I)}{e(I, R)} \right\}_{\sqrt{I}=\fm}$ is bounded from above (see also \cite{MV95}), and they conjectured the following \cite[Theorem 1 and Conjecture]{SV96}:

\begin{conjecture}[St\"{u}ckrad--Vogel]
\label{conj:Vogel}
Let $(R,\fm)$ be a Noetherian local ring and let $M$ be a finitely generated $R$-module. Let $e(I, M)$ be the Hilbert--Samuel multiplicity\footnote{In this paper, we define the Hilbert--Samuel multiplicity of a finitely generated module $M$ with respect to $I$ to be $e(I, M)=\lim\limits_{n\to\infty}t!\frac{l_R(M/I^nM)}{n^t}$ where $t=\dim M$. This is always a positive integer even when $\dim M<\dim R$. We will simplify our notation when $I=\fm$ and write $e(M)$ for $e(\fm, M)$.} of $M$ with respect to $I$. Set $$n(M)=\sup_{\sqrt{I+\Ann(M)}=\fm}\left\{\frac{l(M/IM)}{e(I, M)}\right\}.$$ Then $n(M)<\infty$ if and only if $M$ is quasi-unmixed (i.e., $\widehat{M}$ is equidimensional).
\end{conjecture}

St\"{u}ckrad and Vogel proved the ``only if" direction in general and a graded version of the ``if" direction \cite[Theorem 1]{SV96}. Some other partial results were obtained in \cite{AH00}. In this paper we settle this conjecture in the affirmative. Furthermore, motivated by Conjecture \ref{conj:Vogel}, it is quite natural to inquire whether Lech's classical inequality can be extended to all finitely generated modules, i.e., whether there is a lower bound on the set $\left\{\frac{l(M/IM)}{e(I, M)} \right\}_{\sqrt{I}=\fm}$ for a finitely generated $R$-module $M$. We also answer this question in the affirmative. In sum, our main result is the following:

\begin{maintheoremA}[Theorem \ref{Thm:Vogel} and Theorem \ref{Thm:LechMod}]
Let $(R,\fm)$ be a Noetherian local ring, and let $M$ be a finitely generated $R$-module of dimension $d$. Set
$$m(M)=\inf_{\sqrt{I+\Ann(M)}=\fm}\left\{\frac{l(M/IM)}{e(I, M)} \right\} \text{ and \hspace{0.5em}} n(M)=\sup_{\sqrt{I+\Ann(M)}=\fm}\left\{\frac{l(M/IM)}{e(I, M)} \right\}.$$ Then we have
$$m(M)\geq \frac{1}{d!e(\overline{R})}$$ where $\overline{R}=R/\Ann(M)$. Moreover, if $M$ is quasi-unmixed, then we also have
$$n(M)< \infty.$$
\end{maintheoremA}

As an application of Theorem A, we obtain the following result on Koszul homology:

\begin{maintheoremB}[Theorem \ref{Uniform0}]
Let $R$ be a Noetherian local ring, and let $M$ be a finitely generated quasi-unmixed $R$-module of dimension $d$.  For
every $\varepsilon>0$, there exists $t_0$ such that, for all $t \geq t_0$, all systems of parameters $\ux:= x_1, \ldots, x_d$ of
$M$, and all $1 \leq i \leq d$,
\[
\frac{l(H_i(x_1^t, \ldots, x_d^t; M))}{l(M/(x_1^t, \ldots, x_d^t)M)}
<\varepsilon.
\]
In fact, there exists a constant $K$ such that for all $t \geq
1$, all systems of parameters $\ux:= x_1, \ldots, x_d$ of
$M$, and all $1 \leq i \leq d$,
\[
\frac{l(H_i(x_1^t, \ldots, x_d^t; M))}{l(M/(x_1^t, \ldots, x_d^t)M)}
\le \frac K{t^i}.
\]
\end{maintheoremB}

It is well known that the ratio in Theorem B tends to $0$ for any {\it fixed} system of parameters $\ux$.  What we achieve in Theorem B is a {\it uniform} convergence. We also point out that our Theorem A says that $m(M)$ is bounded below by $\frac{1}{d!e(\overline{R})}$, which is independent of $M$ (and only depends on $R/\Ann(M)$). One cannot expect the same for the upper bound $n(M)$:
\begin{example}
\label{nlarge}
Let $R=k[[x,y]]$ and let $M_t=\fm^t=(x,y)^t$.  Then the $M_t$ are all faithful $R$-modules of rank one, but clearly
$$n(M_t)\geq \frac{l(M_t/\fm M_t)}{e(\fm, M_t)}=\frac{l(\fm^t/\fm^{t+1})}{e(R)}=t+1.$$ Therefore, there cannot exist a constant $c$ such that $n(M_t)\leq c$ works for all $M_t$.
\end{example}
Nonetheless, inspired by this example, we will see in Remark \ref{RemMinGen} that $n(M)/\mu(M)$ is indeed bounded above by a constant depending only on $R/\Ann(M)$.

This paper is organized as follows. In Section 2 we prove Conjecture \ref{conj:Vogel}, which is the second part of the Theorem A, and we also prove some results about the behavior of the invariant $n(M)$ under base change. In Section 3 we extend the classical version of Lech's inequality and  prove the first part of Theorem A. In Section 4 we give many applications of Theorem A, prove Theorem B, and obtain an alternative proof of Conjecture \ref{conj:Vogel}. In the Appendix, we establish global versions of results in Section 4.

\vspace{1em}

\noindent\textbf{Acknowledgement}: The authors would like to thank Mel Hochster and Craig Huneke for valuable discussions, as well as Le Tuan Hoa for his comments. Partial results of this paper were done while the third author was visiting University of Utah. He would like to thank University of Utah for its hospitality and support.  The first author has been partially supported by NSF Grants \#1401384 and \#0943832.  The second author is supported in part by NSF Grant \#1836867/1600198 and NSF CAREER Grant DMS \#1252860/1501102 when preparing this paper. The third author is partially supported by a fund of Vietnam National Foundation for Science and Technology Development (NAFOSTED) under grant number 101.04-2017.10. Finally, the authors would like to thank the referee for his/her comments, which led to Remark \ref{RemMinGen}.

\section{Finiteness of $n(M)$: resolving the St\"{u}ckrad--Vogel conjecture}

To prove the St\"{u}ckrad--Vogel conjecture, we need the concept of extended degree of a finitely generated module introduced by Vasconcelos in \cite{Vas98, VasSixLec98}.

\begin{definition}\label{Def:Edeg}
Let $(R,\fm)$ be a Noetherian local ring with infinite residue field. Let $\mathcal{M}(R)$ denote the category of finitely generated $R$-modules. An {\it extended degree} on $\mathcal{M}(R)$ with respect to an $\fm$-primary ideal $I$ is a numerical function
$$\Deg(I,\bullet)\colon \mathcal{M}(R)\to \mathbb{R}$$ satisfying the following conditions:
\begin{enumerate}
  \item $\Deg(I, M)=\Deg(I, \overline{M})+l(H_{\fm}^0(M))$, where $\overline{M}=M/H_{\fm}^0(M)$;
  \item $\Deg(I, M)\geq \Deg(I, M/xM)$ for every generic element $x\in I-\fm I$ of $M$;
  \item If $M$ is Cohen-Macaulay then $\Deg(I, M)=e(I, M)$.
\end{enumerate}
\end{definition}

The original definition in \cite{VasSixLec98} only deals with the case $I=\fm$. The above definition was taken from \cite[Definition 5.3]{CQ17}. The first question is whether, given a Noetherian local ring $(R,\fm)$, an extended degree function exists. This question was settled in the affirmative by Vasconcelos (\cite{Vas98, VasSixLec98}), who showed that {\it homological degree} is an example of extended degree (when the residue field is infinite).\footnote{Here again, Vasconcelos's papers \cite{Vas98, VasSixLec98} focus on the case $I=\fm$, and in fact the main case Vasconcelos considered is the graded case. However the proofs in \cite{Vas98, VasSixLec98} work in the general set up, and we refer to \cite{CQ17} for more details.}

\begin{definition}
Let $(R,\fm)$ be a homomorphic image of a Gorenstein local ring $(S,\fn)$ of dimension $n$, and let $M$ be a finitely generated $R$-module of dimension $d$. Then the {\it homological degree}, $\hdeg(I, M)$, of $M$ with respect to an $\fm$-primary ideal $I$ is defined by the following recursive formula: $$\hdeg(I, M)=e(I, M)+\sum_{i=n-d+1}^{n}\binom{d-1}{i-n+d-1}\hdeg(I, \Ext_S^i(M, S)).$$
\end{definition}

We note that the above definition is recursive on dimension since $\dim \Ext_S^i(M, S)\leq n-i<d=\dim M$ for all $i=n-d+1,\dots, n$. For a long time, the homological degree was the only known explicit example of an extended degree. Quite recently in \cite{CQ17}, Cuong and the third author discovered another extended degree, this one defined in terms of the {\it Cohen-Macaulay deviated sequence} $\{U_i(M)\}_{i=0}^{d-1}$ of $M$. Roughly speaking, $U_i(M)$ is the unmixed component of $M/(x_{i+2},\dots,x_d)M$ for a certain carefully chosen system of parameters $x_1,\dots,x_d$ of $M$.  It is shown in \cite[Theorem 4.4]{CQ17} that this unmixed component is independent of the choice of $x_1,\dots,x_d$ as long as $x_1,\dots,x_d$ is a {\it $C$-system of parameters of $M$}, which always exists when $R$ is a homomorphic image of a Cohen-Macaulay local ring. Thus, $\{U_i(M)\}_{i=0}^{d-1}$ is a sequence of finitely generated $R$-modules depending only on $M$. Note that $U_{d-1}(M)$ is just the unmixed component of $M$.  We refer to \cite[Section 4]{CQ17} for more details.

\begin{definition}
Let $(R,\fm)$ be a homomorphic image of a Cohen-Macaulay local ring, let $M$ be a finitely generated $R$-module of dimension $d$, and let $U_i(M)$, $0\leq i\leq d-1$, be the Cohen-Macaulay deviated sequence of $M$. We define the {\it unmixed degree} of $M$ with respect to an $\fm$-primary ideal $I$, denoted $\udeg(I, M)$, as follows: $$\udeg(I, M)=e(I, M)+\sum_{i=0}^{d-1}\delta_{i,\dim U_i(M)}e(I, U_i(M)).$$
\end{definition}

It is shown in \cite[Theorem 5.18]{CQ17} that $\udeg(I,\bullet)$ is an extended degree (when the residue field is infinite). We make an elementary but important observation that, for a fixed finitely generated module $M$, $\hdeg(I, M)$ (resp. $\udeg(I, M)$) is a {\it finite} sum $\sum_i e(I, M_i)$, where $\{M_i\}$ only depends on $M$: this is clear from the definition for $\udeg(I, M)$ and is easily seen by induction for $\hdeg(I, M)$. Therefore, by the associativity formula for multiplicities, for a fixed finitely generated $R$-module $M$, there exists a {\it finite} collection of prime ideals $\Lambda(M)=\Lambda \subseteq \mathrm{Supp(M)}$ (allowing repetition) such that
\begin{equation}
\label{eqn:Finite}
\hdeg(I, M)=\sum_{P\in\Lambda}e(I, R/P), \text{ and similarly for $\udeg(I, M)$}.
\end{equation}

Now we are ready to state and prove our main result in this section. We recall that a finitely generated $R$-module $M$ is called {\it quasi-unmixed} if $\widehat{M}$ is equidimensional. This is equivalent to the condition that $\widehat{\overline{R}}$ be equidimensional where $\overline{R}=R/\Ann(M)$.

\begin{theorem}
\label{Thm:Vogel}
Let $(R,\fm)$ be a Noetherian local ring, and let $M$ be a finitely generated quasi-unmixed $R$-module. Then we have
$$n(M)=\sup_{\sqrt{I+\Ann(M)}=\fm}\left\{\frac{l(M/IM)}{e(I, M)} \right\}<\infty.$$
\end{theorem}
\begin{proof}
By passing to the $\frak m$-adic completion, we can assume that $R$ is a complete local ring and $M$ is equidimensional. We can assume also that the residue field is infinity. We now consider $\Deg(I, M)=\hdeg(I, M)$ (or $\Deg(I, M)=\udeg(I, M)$), which is an extended degree.  Thus, by Definition \ref{Def:Edeg} (2) we know that, for every generic element $x\in I-\fm I$ of $M$, we have $$\Deg(I, M)\geq \Deg(I, M/xM).$$ Therefore, for a generic sequence of elements $x_1,\dots,x_d$ of $M$ (we may choose $x_i$ sufficiently general such that $x_1,\dots,x_d$ is a system of parameters of $M$), we have
$$\Deg(I, M)\geq \Deg(I, M/x_1M)\geq \cdots \geq \Deg(I, M/(x_1,\dots,x_d)M)=l(M/(x_1,\dots,x_d))M \geq l(M/IM),$$
where the equality is because $M/(x_1,\dots,x_d)M$ is Cohen-Macaulay and, thus
$$\Deg(I, M/(x_1,\dots,x_d)M)=e(I, M/(x_1,\dots,x_d)M)=l(M/(x_1,\dots,x_d)M).$$
Thus, it is enough to prove that
$$\sup_{\sqrt{I}=\fm}\left\{\frac{\Deg(I, M)}{e(I, M)}\right\}<\infty.$$
At this point we invoke (\ref{eqn:Finite}): it is enough to prove that, for every $P\in\Supp (M)$,
\begin{equation}\label{eqn:supR/P}
\sup_{\sqrt{I}=\fm}\left\{\frac{e(I, R/P)}{e(I, M)}\right\}<\infty.
\end{equation}

In order to prove (\ref{eqn:supR/P}), we use induction on $\dim M$. If $\dim M=0$, (\ref{eqn:supR/P}) is obvious. In the general case, if $\dim R/P=\dim M$ then $e(I, R/P) \le e(I, M)$ by the associativity of multiplicities, so (\ref{eqn:supR/P}) is again obvious. Now we assume $\dim R/P<\dim R$.  We choose a prime ideal $P_0 \in \Ass M$ such that $\dim R/P_0 =\dim M$ and $P_0 \subseteq P$. We have $e(I, R/P_0) \le e(I,M)$ by the associativity of multiplicities again. Therefore, it is enough to prove that, for every $P\in\Spec R$,
\begin{equation}\label{eqn:supR/P}
\sup_{\sqrt{I}=\fm}\left\{\frac{e(I, R/P)}{e(I, R)}\right\}<\infty,
\end{equation}
where $R$ is a complete local domain and $\dim R/P < \dim R$. We pick $0\neq x\in P$ and a minimal prime $Q$ of $(x)$ such that $Q\subseteq P$. Since $R$ is a complete local domain, $R/(x)$ is equidimensional; in particular, $\dim R/(x)=\dim R/Q$, and, thus, $e(I, R/(x))\geq e(I, R/Q)$. Now we write
$$
\frac{e(I, R/P)}{e(I, R)}=\frac{e(I, R/P)}{e(I, R/Q)}\cdot\frac{e(I, R/Q)}{e(I, R/(x))}\cdot \frac{e(I, R/(x))}{e(I, R)}\leq \frac{e(I, R/P)}{e(I, R/Q)}\cdot\frac{e(I, R/(x))}{e(I, R)}.
$$

Since $\dim R/Q<\dim R$, $\sup_{\sqrt{I}=\fm}\left\{\frac{e(I, R/P)}{e(I, R/Q)}\right\}<\infty$ by induction, which means there exists a constant $c_1$ such that $\frac{e(I, R/P)}{e(I, R/Q)}\leq c_1$ for all $\fm$-primary ideals $I$. Since $x$ is a nonzerodivisor in a complete local ring $R$, by Lemma \ref{lem:Huneke} below, we know that there exists a constant $c_2$ such that $\frac{e(I, R/(x))}{e(I, R)}\leq c_2$ for all $\fm$-primary ideals $I$. Thus, putting $c=c_1c_2$ we see that
$$
\frac{e(I, R/P)}{e(I, R)}\leq c
$$
for all $\fm$-primary ideals $I$. This finishes the proof.
\end{proof}

\begin{lemma}\label{lem:Huneke}
Let $(R,\fm)$ be a Noetherian complete local ring, and let $x$ be a nonzerodivisor on $R$. Then there exists a constant $k$ such that, for all $\fm$-primary ideals $I$, we have $$e(I, R/(x))\leq k\cdot e(I, R).$$
\end{lemma}
\begin{proof}
We consider the short exact sequence:
$$
0\to \frac{R}{I^n:x} \xrightarrow{\cdot x}\frac{R}{I^n}\to \frac{R}{I^n+(x)}\to 0
$$
Note that if $y\in I^n:x$, then $xy\in I^n\cap(x)$. By Huneke's uniform Artin-Rees lemma \cite[Theorem 4.12]{Hun92}, there exists a constant $k$ such that, for all $I\subseteq R$, $I^n\cap(x)\subseteq I^{n-k}x$. Thus, $xy\in I^{n-k}x$, and so $y\in I^{n-k}$ since $x$ is a nonzerodivisor. This shows that $I^n:x\subseteq I^{n-k}$ for all $\fm$-primary ideals $I$. By the short exact sequence above, we know that
$$
l\left(\frac{R}{I^n+(x)}\right)\leq l\left(\frac{R}{I^n}\right)-l\left(\frac{R}{I^{n-k}}\right)
$$
Now we let $n\to\infty$ and compute the corresponding Hilbert function to see that
$$e(I, R/(x))\leq k\cdot e(I, R)$$
for all $\fm$-primary ideals $I$.
\end{proof}

\begin{remark}\label{GlobalSV}
Assume that $R$ is a (not necessarily local) Noetherian ring such that
$R$ is a homomorphic image of a Noetherian
Gorenstein ring $S$ with $\dim(S) < \infty$. Further assume that $M$ is a finitely generated $R$-module such that $R/\Ann(M)$ is locally equidimensional and satisfies the uniform
Artin-Rees property (e.g., \cite[Theorem 4.12]{Hun92}) and that, for all $P \in \Supp(M)$, the residue field of $R_P$ is infinite. Then the proof of Theorem~\ref{Thm:Vogel} actually gives us a (global) upper bound of $n(M_P)$ for all $P \in \Supp(M)$. Also see Remark~\ref{SVbyInduction-g} for an alternative treatment.
\end{remark}

Given Theorem \ref{Thm:Vogel}, it is quite natural to ask whether the supremum is actually a maximum, i.e., whether $n(M)$ is attained at some $I$. We do not know the answer to this question. Below we prove a special case. Recall that a finitely generated $R$-module $M$ is called {\it generalized Cohen-Macaulay} if $H_{\fm}^i(M)$ has finite length for all $i<\dim (M)$.

\begin{theorem}
\label{max}
Let $(R,\fm)$ be a Noetherian local ring with infinite residue field, and let $M$ be a finitely generated $R$-module.
If $M$ is generalized Cohen-Macaulay (e.g., $\dim(M) = 1$), then $n(M)$ is attained, i.e., $n(M) = l(M/IM)/e(I,M)$ for some
$\fm$-primary ideal $I$.
\end{theorem}
\begin{proof}
Since the residue field of $R$ is infinite, every $\fm$-primary ideal $I$ has a minimal reduction $(\ux)$ generated by a system of parameters of $M$. Because $\frac{l(M/IM)}{e(I, M)}\leq \frac {l(M/(\ux)M)}{e((\ux),M)}$ for any minimal reduction $(\ux)$ of $I$, we have
\[
n(M) = \sup\left\{\frac {l(M/(\ux)M)}{e((\ux),M)} \ \Big | \
\ux \text{ is a system of parameter of } M\right\}.
\]
When $M$ is Cohen-Macaulay, it is easy to see that $n(M) = 1 = \frac{l(M/(\ux)M)}{e((\ux),M)}$ for any ideal $(\ux)$ generated by a system of parameters of $M$. Thus, we assume that $M$ is not Cohen-Macaulay; hence, $n(M) > 1+\varepsilon$ for some $\varepsilon > 0$.

Since $M$ is generalized Cohen-Macaulay, it is well known that there exists $C > 0$ (e.g., $C = \sum_{i=0}^{d-1}\binom{d-1}{i}l(H_\fm^{i}(M))$ \cite[Theorem 3.18]{Schenzel98})
such that
\[
l(M/(\ux)M)  \le e((\ux),M) + C
\qquad \text{hence} \qquad
\frac {l(M/(\ux)M)}{e((\ux),M)} \le 1 + \frac C{e((\ux),M)}
\]
for all systems of parameters $\ux$ of $M$ (for example see \cite{Schenzel83} or \cite{SV78}). This shows that $n(M) \le
1 + \frac C{e(M)} < \infty$. There exists a positive
integer $N$ such that $1+ \frac Cn < 1+ \varepsilon$ for all $n > N$.
Therefore
\[
n(M) = \sup\left\{\frac {l(M/(\ux)M)}{e((\ux),M)} \ \Big | \
\ux \text{ is a system of parameters of $M$ such that }
e((\ux),M) \le N \right\}.
\]
However, the set of numbers on the right side is finite, so $n(M)$ is
must be attained at some system of parameters of $M$.
\end{proof}

\subsection{The behavior of $n(M)$ under base change} In this subsection we study the behavior of $n(M)$ under localization, flat local extension, and the killing of a parameter. We begin with a result on localization.

\begin{theorem}
\label{Thm:localization}
Let $(R, \fm)$ be a Noetherian local ring and $M$ a finitely generated $R$-module. Then for any $P \in \Supp (M)$, we have $n_{R_P}(M_P) \leq n_R(M)$. 
\end{theorem}
\begin{proof}
We can assume $M$ is quasi-ummixed since otherwise $n(M)=\infty$ by \cite[Theorem 1]{SV78}, in which case there is nothing to prove. We can replace $R$ by $R/\Ann(M)$: this does not affect $n(M)$ or $n(M_P)$. Therefore, we can assume $R$ is quasi-unmixed. By \cite[Theorem 31.6]{Matsumura86}, this implies $R$ is equidimensional and catenary and that $R_P$ is also quasi-unmixed. We set $\dim (R)=d$. 


By induction it is enough to consider the case of $\dim (R/P)=1$. Since the residue field of $R_P$ is infinite, it suffices to show that
\[
n(M) \geq \frac{\length_{R_P}(M_P/IM_P)} { e(I, M_P)}\]
for all ideals $I$ generated by a system of parameters in $R_P$ (as in the proof of Theorem \ref{max}). We know $\dim R_P=\Ht P=d-1$ since $R$ is equidimensional and catenary; thus, for any such $I$, by prime avoidance, we can find elements $x_1, \ldots, x_{d-1} \in R$
 that  form part of a system of parameters in $R$ and that have images in $R_P$ that generate $I$.
So, abusing notation a bit, we will call $I = (x_1, \ldots, x_{d-1}) \subseteq R$.

Suppose $x\in R$ is such that $(I, x)$ is $\fm$-primary. Since $M$ is faithful, $x_1,\dots,x_{d-1}, x$ form a system of parameters on both $R$ and $M$. We have
\[
\length(M/(I, x)M) \geq e(x, M/IM)  = \sum_{Q \in \Min(M/IM)} e(x, R/Q)\length(M_Q/IM_Q),
\]
where the equality holds by the additivity property of multiplicity. By Lech's associativity formula for multiplicities for parameter ideals \cite{Lech57} (see also \cite{Nagata56}), we also have
\[
e ((I,x), M) = \sum_{Q \in \Min(M/IM)} e(x, R/Q)e(IR_Q, M_Q).
\]
Therefore, by definition,
\begin{align*}
n(M) &\geq \frac{\length(M/(I,x)M)}{e((I,x), M)}
\geq \frac{\sum_{Q \in \Min(M/IM)} e(x, R/Q)\length_{R_Q} (M_Q/IM_Q)}{\sum_{Q \in \Min(M/IM)}e(x, R/Q)e(IR_Q, M_Q)}.
\end{align*}
Now we let $ \displaystyle y \in \bigcap_{\substack{Q\in \Min(M/IM)\\Q \neq P}} Q \setminus P$ and $\displaystyle z \in P \setminus \bigcup_{\substack{Q\in \Min(M/IM)\\Q \neq P}}  Q$.
Observe that, for any $t \geq 1$, we can use $x=y^t + z$ to complete $I$ to a full system of parameters. In this case
\begin{equation*}
\frac{\sum_{Q} e(x, R/Q)\length_{R_Q} (M_Q/IM_Q)}{\sum_{Q}e(x, R/Q)e(IR_Q, M_Q)}
= \frac
{e(y^t, R/P)\length_{R_P} (M_P/IM_P) + \sum_{Q \neq P} e(z, R/Q)\length_{R_Q} (M_Q/IM_Q)}
{e(y^t, R/P)e(IR_P, M_P) + \sum_{Q \neq P}e(z, R/Q)e(IR_Q, M_Q)}.
\end{equation*}
Since $e(y^t, R/P) = t e (y, R/P)$, if we pass to the limit as $t$ approaches infinity we obtain
\[
n(M) \geq \frac{\length_{R_P} (M_P/IM_P)}{e(IR_P, M_P)}.
\qedhere
\]
\end{proof}

As a consequence, we show that the invariant $n(-)$ is non-decreasing under flat local extensions.

\begin{corollary}
\label{cor.flat}
Let $(R, \mf m)\to (S,\mf n)$ be a flat local extension of Noetherian local rings. Suppose $M$ is a finitely generated $R$-module. Then $n_R(M) \leq n_S(M \otimes_R S)$.
\end{corollary}
\begin{proof}
Let $P$ be a minimal prime of $\mf mS$.  By Theorem \ref{Thm:localization}, we have
\[n_{S_P}((M \otimes_R S)_P) \leq n_S(M \otimes_R S).\]
Thus, replacing $S$ by $S_P$, we may assume that $S$ is local and that $\mf mS$ is $\mf n$-primary.
For any $\mf m$-primary ideal $I$, its extension $IS$ is an $\mf n$-primary ideal, and
tensoring the composition series with $S$ shows that
\[
\length_R (M/IM)\length_S (S/\mf mS) = \length_S \left((M\otimes_R S)/I(M\otimes_R S)\right)
\]
for any finitely generated $R$-module $M$.
Thus, $e (I, M)\length_S (S/\mf mS) = e(IS, M\otimes_R S)$ and
\begin{align*}
n(M) &= \sup_{\sqrt{I+\Ann(M)}=\fm} \left\{\frac{\length_R (M/IM)}{e(I, M)}\right\} \\
&=
\sup_{\sqrt{I+\Ann(M)}=\fm} \left\{ \frac{\length_S \left ((M \otimes_R S)/I(M \otimes_R S)\right )}{e (IS, M\otimes_RS)} \right\} \\
   &\leq  \sup_{\sqrt{IS+\Ann(M)S}=\fn} \left\{ \frac{\length_S \left ((M \otimes_R S)/I(M \otimes_R S)\right )}{e (IS, M\otimes_RS)} \right\}= n(M \otimes_R S). \qedhere
\end{align*}
\end{proof}


Given a flat local extension $(R,\fm)\to (S, \fn)$ and a finitely generated $R$-module $M$, it would also be interesting to bound $n(M\otimes_RS)$ in terms of $n(M)$ and $n(S/\fm S)$. We do not know how to obtain such a relation yet. Our last result in this subsection relates $n(M)$ and $n(M/xM)$ for a parameter $x$ on $M$ (i.e., $\dim (M/xM)=\dim(M)-1$ or, equivalently, $x$ is a parameter on $R/\Ann(M)$).

\begin{proposition}
Let $(R, \mf m)$ be a Noetherian local ring, and let $M$ be a finitely generated $R$-module of dimension $d$.
Then for any parameter $x$ of $M$, we have $n(M/xM) \leq n(M)$.
\end{proposition}
\begin{proof}
Replacing $R$ by $R/\Ann(M)$, we may assume $M$ is a faithful $R$-module. Hence $x$ is a parameter on $R$ as well.
Since $e(I, M/xM)\geq e(I, M)$ for every $\fm$-primary ideal $I$ that contains $x$, we have
\[
\frac{\len((M/xM)/I(M/xM))}{e(I, M/xM)} =\frac{l(M/IM)}{e(I, M/xM)}
\le \frac{\len(M/IM)}{e(I, M)} \leq n(M).
\]
This clearly implies $n(M/xM) \leq n(M)$, as desired.
\end{proof}

\section{The lower bound: a generalization of Lech's inequality}

Our goal in this section is to generalize Lech's inequality to all finitely generated $R$-modules, thus proving the first part of Theorem A in the introduction. We first prove a key lemma.

\begin{lemma}
\label{lem:CDHZ}
Let $(R,\fm, k)$ be a complete local domain with an algebraically closed residue field. Let $M$ be a finitely generated $R$-module with $\dim(R) = \dim(M)$, and let $J$ be an integrally closed $\fm$-primary ideal. Then we have $$l(M/JM)\geq l(R/J)\dim_K(M\otimes_RK),$$ where $K$ denotes the fraction field of $R$.
\end{lemma}
\begin{proof}
First of all, if we let $T(M)$ denote the torsion submodule of $M$, then we have $$0\to T(M)\to M\to M'\to 0$$ where $M'$ is torsion-free. Since $l(M/JM)\geq l(M'/JM')$ while $\dim_K(M\otimes_RK)=\dim_K(M'\otimes K)$, if the lemma holds for $M'$ then it also holds for $M$. Thus, in the rest of the proof we assume $M$ is torsion-free. In this case $\dim_K(M\otimes K)=\rank M$.

By \cite[Corollary 2.2]{CDHZ12}, we have
$$l(M/JM)\geq \bar{l}(R/J)\cdot\rank M,$$
where $\bar{l}(R/J)$ denotes the length of the longest chain of integrally closed ideals between $J$ and $R$. Therefore, it is enough to show $\bar{l}(R/J)=l(R/J)$. To prove this it is enough to find an integrally closed ideal $J'\supseteq J$ in $R$ such that $l(J'/J)=1$ because then $\bar{l}(R/J)=l(R/J)$ follows from an easy induction. Let $R\to S$ be the normalization of $R$. Since $R$ is a complete local domain, $S$ is local by \cite[Proposition 4.8.2]{HS06}, and so $S=(S,\fn)$ is a normal local domain with $R/\fm=S/\fn=k$ since $k$ is algebraically closed. Now by \cite[Theorem 2.1]{Wat03}, there exists a chain $$\overline{JS}=J_0\subseteq J_1\subseteq J_2\subseteq\cdots \subseteq J_n=\fn$$ such that
\begin{enumerate}
  \item Each $J_i$ is integrally closed in $S$;
  \item $l(J_{i+1}/J_i)=1$ for every $i$.
\end{enumerate}
Since $J$ is integrally closed in $R$ and $S$ is integral over $R$, by \cite[Proposition 1.6.1]{HS06} we know
$$J_0\cap R=\overline{{JS}}\cap R=\bar{J}=J.$$
Let $t=\max\{i \mid J_i\cap R=J\}$.  Obviously $0\leq t<n$. Set $J'=J_{t+1}\cap R$.  It is easy to see that $J'\supseteq J$ is integrally closed in $R$ (one can use \cite[Proposition 1.6.1]{HS06} again). Moreover, $l(J'/J)>0$ by our choice of $t$ while $J'/J\hookrightarrow J_{t+1}/J_t$ shows that $l(J'/J)\leq l(J_{t+1}/J_t)=1$. Thus, we have $l(J'/J)=1$.
\end{proof}

We define $\Assh(M) = \{P \in \Ass(M) \mid \dim(R/P) = \dim(M)\}$ for a finitely generated $R$-module $M$. We are now ready to state and prove the following generalization of Lech's inequality.
\begin{theorem}
\label{Thm:LechMod}
Let $(R,\fm, k)$ be a Noetherian local ring, and let $M$ be a finitely generated $R$-module of dimension $d$. Then for every ideal $I$ of $R$ whose image in $\overline{R}=R/\Ann(M)$ is $\fm$-primary, we have $$e(I, M)\leq d!e(\overline{R})l(M/IM).$$
\end{theorem}
\begin{proof}
Replacing $R$ by $\overline{R}$ does not change either side of the inequality. Therefore, we may assume $\Ann(M)=0$ and so that $\dim(R)=\dim(M)=d$. We next take a flat local extension $(R,\fm,k)\to (R', \fm',k')$ such that $\fm'=\fm R'$ and $k'=R'/\fm'$ is the algebraic closure of $R/\fm=k$.  (Such an $R'$ always exists: it is a suitable gonflement of $R$; see \cite[Corollaire in Appendice 2]{Bou06}). Then $R\to R'\to \widehat{R'}$ is a faithfully flat extension with $\fm\widehat{R'}=\fm_{\widehat{R'}}$, so passing from $R$ to $\widehat{R'}$ and replacing $M$ by $M\otimes_R \widehat{R'}$ do not affect either side of the inequality. Therefore, without loss of generality, we may assume $(R,\fm, k)$ is a complete local ring with $k=\bar{k}$ and $\Assh(R)=\Assh(M)$.

By the associativity formula of multiplicity, we have
$$e(I, M)=\sum_{P\in \Assh(M)}l_{R_P}(M_P)e(I, R/P)=\sum_{P\in \Assh(R)} l_{R_P}(M_P)e(\overline{IR/P}, R/P).$$
Using Lech's inequality \cite[Theorem 3]{Lech60} for each $R/P$, we have
\begin{equation}
\label{eqn:Associativity}
e(I, M)\leq \sum_{P\in \Assh(R)} d!e(R/P)l\left((R/P)/(\overline{IR/P})\right)l_{R_P}(M_P).
\end{equation}
\begin{claim}
For every minimal prime $P$ of $R$, we have
\begin{equation}\label{eqn:Claim}
l\left((R/P)/(\overline{IR/P})\right)\cdot l_{R_P}(M_P)\leq l(M/IM)\cdot l_{R_P}(R_P).
\end{equation}
\end{claim}
\begin{proof}[Proof of Claim]
Clearly we have $l_{R_P}(M_P)\leq l_{R_P}(R_P)\cdot l_{R_P}(M_P/PM_P)$ because $l_{R_P}(M_P/PM_P)$ is the minimal number of generators of $M_P$ as an $R_P$-module. Therefore,
$$l\left((R/P)/(\overline{IR/P})\right)\cdot l_{R_P}(M_P)\leq l\left((R/P)/(\overline{IR/P})\right)\cdot l_{R_P}(M_P/PM_P)\cdot l_{R_P}(R_P).$$
Now $R/P$ is a complete local domain with algebraically closed residue field $k=\bar{k}$, and $M/PM$ is a finitely generated $R/P$-module. Applying Lemma \ref{lem:CDHZ} and noting that $\dim_{\kappa(P)}(M/PM)\otimes \kappa(P)=l_{R_P}(M_P/PM_P)$, we have
$$l\left((R/P)/(\overline{IR/P})\right)\cdot l_{R_P}(M_P/PM_P) \leq l\left( \frac{M/PM}{(\overline{IR/P})(M/PM)}\right)\leq l\left(\frac{M}{(I+P)M}\right)\leq l(M/IM).$$ Putting the two inequalities above together, we get $$l\left((R/P)/(\overline{IR/P})\right)\cdot l_{R_P}(M_P)\leq l(M/IM)\cdot l_{R_P}(R_P).$$ This finishes the proof of the Claim.
\end{proof}
Finally, we plug in (\ref{eqn:Claim}) to (\ref{eqn:Associativity}) and use additivity of multiplicity to get
\begin{align*}
  e(I, M) &\leq  \sum_{P\in \Assh(R)} d! e(R/P) l_{R_P}(R_P)l(M/IM) \\
   &= d! l(M/IM) \left(\sum_{P\in \Assh(R)} l_{R_P}(R_P) e(R/P)\right)
   = d! e(R) l(M/IM).
\end{align*}
This finishes the proof.
\end{proof}

\section{Applications and an alternative approach}

In this section we give some applications of our results in Section 2 and Section 3. In the process, we obtain another way to prove Conjecture \ref{conj:Vogel} without using Vasconcelos's extended degree (see Remark~\ref{SVbyInduction}).

\begin{lemma}
\label{lem:Length}
If $(R, \fm)$ is a Noetherian local ring and $M$ is a finitely generated quasi-unmixed $R$-module of dimension $d$, then there exists a constant $C$ such that
for every $h \in \mathbb R$
\[
\frac hC \le \inf_{\substack{\sqrt{I}= \sqrt{J} =\fm \\ e(I, M) \ge h e(J,M)}}
\left\{\frac{l(M/IM)}{l(M/JM)} \right\}
\quad \text{and} \quad
\sup_{\substack{\sqrt{I}= \sqrt{J} =\fm \\ e(I, M) \le h e(J,M)}}
\left\{\frac{l(M/IM)}{l(M/JM)} \right\} \le hC.
\]
In particular, there exists a constant $C$ such that such that for all $\frak m$-primary ideals $I$ and for all $n \ge 1$ we have
$$\frac{n^d}{C} \le \frac{l(M/I^nM)}{l(M/IM)} \le Cn^d.$$
\end{lemma}

\begin{proof}
Let $\dim(M) = \dim (\overline{R})=d$ where $\overline{R}=R/\Ann(M)$. We use Theorem \ref{Thm:Vogel} and Theorem \ref{Thm:LechMod} to see
\[
l(M/IM) \geq  e(I,M)/(d! e(\overline{R})) \ge h  e(J,M)/(d! e(\overline{R}))
\geq h l(M/JM) /(n(M)  d! e(\overline{R})),
\]
which proves the first inequality. By symmetry, we have
\[
l(M/IM) \leq n(M) e(I,M) \le n(M)h  e(J,M) \leq h n(M)  d! e(\overline{R})  l(M/JM),
\]
which proves the second inequality.
So we can take $C=n(M)  d! e(\overline{R})$ in both cases. Finally, taking $h=n^d$ and noting that $e(I^n,M) = n^d e(I,M)$ immediately proves the last claim.
\end{proof}

This lemma has an immediate consequence, which we will need.

\begin{corollary}
\label{cor.integralclosure}
If $(R,\fm)$ is a Noetherian local ring and $M$ is a finitely generated quasi-unmixed $R$-module, then
$$\sup_{\substack{ \sqrt{I}=\fm \\ I\subseteq J\subseteq \overline{I}}}
\left\{\frac{l(M/IM)}{l(M/JM)} \right\}<\infty.$$
\end{corollary}
\begin{proof}
The condition $I\subseteq J\subseteq \overline{I}$ implies $e(I, M)=e(J, M)$. Therefore we apply Lemma \ref{lem:Length} with $h=1$ to get the desired claim.
\end{proof}

Next we prove a result that extends Lemma \ref{lem:CDHZ}
at the cost of precision in the inequality.

\begin{lemma}\label{lem:Length2}
Let $(R, \fm)$ be a Noetherian local ring and $N$ be a
finitely generated quasi-unmixed $R$-module.
Then there exists a constant $C_N>0$ depending only on $N$ such that
\[
1/C_N \le \inf_{\sqrt{I} =\fm}\left\{\frac{l(M/IM)}{l(N/IN)} \right\}
\quad \text{and, equivalently,} \quad
\sup_{\sqrt{I} =\fm}\left\{\frac{l(N/IN)}{l(M/IM)} \right\} \le C_N
\]
for all finitely generated $R$-modules $M$ with $\Supp(M) \supseteq \Supp(N)$.
\end{lemma}

\begin{proof}
Since $N$ is quasi-unmixed, it is equidimensional. We set $c = \max_{P \in \Min(N)} l_{R_P}(N_P)$.
Let $M$ be any finitely generated $R$-module $M$ with $\Supp(M) \supseteq
\Supp(N)$. Denote $\overline M = M/\fa M$, where $\fa = \cap_{P \in
  \Min(N)}P$.
Then $\Supp(\overline M) = \Supp(N)$, and,
by the associativity formula for multiplicities,
\begin{align*}
ce(I,\overline M) =\sum_{P \in \Min(N)} c \cdot l_{R_P}(\overline M_P) e(I,R/P)
\ge \sum_{P \in \Min(N)} l_{R_P}(N_P) e(I,R/P)
= e(I,N)
\end{align*}
for all $\fm$-primary ideals $I$.
Now we use Theorem~\ref{Thm:Vogel} and Theorem~\ref{Thm:LechMod} to
obtain (with $d = \dim(N)$)
\[
l(N/IN) \le n(N) e(I,N) \le n(N)c e(I, \overline M)
\le n(N)c d! e(\overline{R}) l(\overline M/I\overline M)
\le n(N)c d! e(\overline{R}) l(M/IM)
\]
where $\overline{R} = R/\fa$ depends only on $N$.
So we can take $C_N={n(N)cd!e(\overline R)}$.
\end{proof}

Lemma~\ref{lem:Length2} allows us to establish the following general
result:

\begin{theorem}\label{thm:Length2}
Let $(R, \fm)$ be a Noetherian local ring,  and let $M$ and $N$ denote
finitely generated $R$-modules. Then
\begin{enumerate}
\item $0 < \inf_{\sqrt{I} =\fm}\left\{\frac{l(M/IM)}{l(N/IN)} \right\}
\iff \sup_{\sqrt{I} =\fm}\left\{\frac{l(N/IN)}{l(M/IM)} \right\} < \infty
\iff \Supp(M) \supseteq \Supp(N)$.
\item There exists a constant $C>0$ depending only on $N$ such that
\[
1/C \le \inf_{\sqrt{I} =\fm}\left\{\frac{l(M/IM)}{l(N/IN)} \right\}
\quad \text{and, equivalently,} \quad
\sup_{\sqrt{I} =\fm}\left\{\frac{l(N/IN)}{l(M/IM)} \right\} \le C
\]
for all (finitely generated $R$-modules) $M$ with $\Supp(M) \supseteq \Supp(N)$.
\item $0 < \inf_{\sqrt{I} =\fm}\left\{\frac{l(M/IM)}{l(N/IN)} \right\}
\le \sup_{\sqrt{I} =\fm}\left\{ \frac{l(M/IM)}{l(N/IN)} \right\} < \infty
\iff \Supp(M) = \Supp(N)$.
\end{enumerate}
\end{theorem}

\begin{proof}
(1): Clearly, we only need to prove the second equivalence. For the
forward direction, assume
$\sup_{\sqrt{I} =\fm}\left\{\frac{l(N/IN)}{l(M/IM)} \right\} <
\infty$ and let $P \in \Supp(N)$. Denote $\ol M = M/PM$ and $\ol N =
N/PN$. As $\left\{\frac{l(N/IN)}{l(M/IM)} \mid \sqrt{I} =\fm\right \}
\supseteq \left\{\frac{l(\ol N/I\ol N)}{l(\ol M/I\ol M)}
\mid \sqrt{I} =\fm\right \}$, we get
$\sup_{\sqrt{I} =\fm}\left\{\frac{l(\ol N/I\ol N)}{l(\ol M/I\ol M)} \right\} <
\infty$, which implies $\dim(\ol N) \le \dim(\ol M)$ by considering
$\frac{l(\ol N/I\ol N)}{l(\ol M/I\ol M)}$ with $I = \fm^t$ for $t \gg 0$.
Note that $\dim(N/PN) = \dim(R/P)$, since $P \in \Supp(N)$. Thus,
$\dim(M/PM) = \dim(R/P)$, which forces $P \in \Supp(M)$.

For the backward direction, assume $\Supp(M) \supseteq \Supp(N)$. We can
further assume that $R$ is complete, which does not affect the
statement. We next take a prime cyclic filtration of $N$ of length $n$ with factors $N_i = R/P_i$ such that $P_i \in
\Supp(N)$ for $i = 1,\dotsc, n$ (note that the $P_i$ are not necessarily distinct). As $l(N/IN) \le
\sum_{i=1}^nl(N_i/IN_i)$ for every $\fm$-primary ideal $I$, it
suffices to show $\sup_{\sqrt{I} =\fm}\left\{\frac{l(N_i/IN_i)}{l(M/IM)} \right\} <
\infty$ for each $i = 1, \dotsc, n$. But this follows from
Lemma~\ref{lem:Length2} since each $N_i = R/P_i$ is quasi-unmixed (since
$R$ is complete) and $\Supp(M) \supseteq \Supp(N_i)$. In detail, let
$C_{N_i} > 0$ be as in Lemma~\ref{lem:Length2} for each $i = 1, \dotsc, n$. Then
\[
\sup_{\sqrt{I} =\fm}\left\{\frac{l(N/IN)}{l(M/IM)} \right\}
\le \sum_{i=1}^n\sup_{\sqrt{I} =\fm}\left\{\frac{l(N_i/IN_i)}{l(M/IM)}
\right\}
\le \sum_{i=1}^n C_{N_i} < \infty
\]
with $\sum_{i=1}^n C_{N_i} < \infty$ depending only on $N$.

(2): From the proof of (1) above (for the backward direction), we can set $C
= \sum_{i=1}^n C_{N_i}$, which depends only on (the completion of) $N$.

(3): This is clear from (1).
\end{proof}

\begin{lemma}
\label{ses}
Let $(R,\fm)$ be a Noetherian local ring and $M$ a finitely generated $R$-module.
If $(y_1, \ldots, y_d) \subseteq (x_1, \ldots, x_k)$ are $\fm$-primary
ideals of $R$, then for all $0 \leq i \leq k$,
\[
l(H_i(x_1, \ldots, x_k;M))
\leq \sum_{j=0}^k {k \choose j} l(H_{i-j}(y_1, \ldots, y_d;M))
\le 2^k \max_{0 \le j \le k}  l(H_{i-j}(y_1, \ldots, y_d;M)),
\]
with the convention that $H_{<0}(y_1, \ldots, y_d; M)=0$.
\end{lemma}

\begin{proof}
If $\underline{f} = f_1, \ldots, f_s$ is any sequence of elements of
$R$ and $\underline{f}^- = f_1, \ldots, f_{s-1}$, then there is
a short exact sequence for each $0 \leq i \leq s-1$
\begin{equation*}
0 \rightarrow \dfrac{H_i(\underline{f}^-;M)}{f_s H_i(\underline{f}^- ;M)} \rightarrow H_i(\underline{f};M) \rightarrow \Ann_{H_{i-1}(\underline{f}^-;M)}(f_s) \rightarrow 0,
\end{equation*}
Using the short exact sequence above, we see that
\begin{align*}
l(H_i(x_1, \ldots, x_k;M))
&\leq l(H_i(x_1, \ldots, x_k;M)) + l(H_{i-1}(x_1, \ldots, x_k;M))\\
&= l(H_i(x_1, \ldots, x_k, y_1;M))
\quad \text{ (since $y_1 \in (x_1, \ldots, x_k)$)}\\
&\leq \dotsb \quad \text{ (by joining $y_2, \dotsc, y_d$ inductively)} \\
&\leq l(H_i(x_1, \ldots, x_k, y_1, \ldots, y_d;M))\\
&= l\left(\dfrac{H_i(x_1, \ldots, x_{k-1}, y_1, \ldots, y_d;M)}
{x_{k} H_i(x_1, \ldots, x_{k-1}, y_1, \ldots, y_d;M)}\right)
+ l\left(\Ann_{H_{i-1}(x_1, \ldots, x_{k-1}, y_1, \ldots, y_d;M)}(x_k)
\right) \\
&\leq l\left(H_i(x_1, \ldots, x_{k-1}, y_1, \ldots, y_d;M)\right)
+ l\left(H_{i-1}(x_1, \ldots, x_{k-1}, y_1, \ldots, y_d;M)\right) \\
&\leq \dotsb \quad \text{ (by removing $x_{k-1}, \dotsc, x_1$ inductively)} \\
&\leq \sum_{j=0}^k {k \choose j} l(H_{i-j}(y_1, \ldots, y_d;M))\\
&\leq 2^k \max_{0 \le j \le k} l(H_{i-j}(y_1, \ldots, y_d;M)),
\end{align*}
completing the proof.
\end{proof}

\begin{remark}[\cite{Schenzel98} or \cite{BHM18}]
\label{sseq}
Let $(R,\fm)$ be a local ring of dimension $d$ that is a homomorphic
image of a local Gorenstein ring $S$ of dimension $n$. Then for every
finitely generated $R$-module $M$, every system of parameters $\underline{x} = x_1, \dotsc, x_d$ of
$R$, and every $i = 1, \dotsc, d$, we have
\[
l(H_i(x_1, \ldots, x_d;M)) \leq
\sum_{j = 0}^{d-i} l(H_{d-i-j}(x_1, \ldots, x_d;\Ext_S^{n-j}(M,S))).\footnote{This is written down in \cite{BHM18}. We point out that this also follows from \cite[Theorem 3.16]{Schenzel98} as follows: since $\underline{x}$ is a system of parameters of $R$, we can pick $\underline{y}=y_1,\dots,y_d$ with $(\underline{y})=(\underline{x})$ such that they form a strong filter regular sequence on $R$ and $M$ by prime avoidance.  Replacing $\underline{x}$ by $\underline{y}$ does not affect the Koszul homology, so we can assume $\underline{x}$ is a strong filter regular sequence on $R$ and $M$ and then note that we have a canonical isomorphism $H^j(x_1,\dots,x_d; N^\vee)\cong H_j(x_1,\dots,x_d; N)^\vee$ for all finitely generated $R$-modules $N$ by \cite[bottom of page 286]{Schenzel98}. (In particular, they have the same length.).  Therefore, the displayed formula is a restatement of \cite[Theorem 3.16]{Schenzel98}.}
\]
Note that $\dim(\Ext_S^{n-j}(M,S)) \le d-i$ for each $j = 0, \dotsc, d-i$, since $\Ext_S^{n-j}(M,S)^\vee\cong H_{\fm}^j(M)$ where $(-)^\vee$ stands for Matlis dual.
\end{remark}

\begin{theorem} \label{Thm:HomoBound}
Let $(R, \fm)$ be a Noetherian local ring and $M$ a finitely generated $R$-module of dimension $d$.  Then there exists a constant $C$ depending on $M$ such
that, for every $k \geq d$, we have
\[
\sup_{\substack{\sqrt{(x_1,\ldots, x_{k}) + \Ann(M)}=\fm\\ 0 \leq i \leq k}}
\left\{\frac{l(H_i(x_1, \ldots, x_{k};M))}
{l(M/(x_1,\ldots,x_{k})M)} \right\} \le 2^kC.
\]
\end{theorem}

\begin{proof}
As in the first paragraph in the proof of Theorem \ref{Thm:LechMod}, we can replace $R$ by $R/\Ann(M)$, enlarge the residue field of $R$, and then complete $R$. Therefore, we can assume that $R$ is a complete local ring with infinite residue field and that $M$ is a faithful $R$-module. (The rest of the proof
only relies on the fact that $(R,\fm)$ is a homomorphic image of a Gorenstein ring $S$ with infinite residue field.)

We proceed by induction on $d= \dim(M)$. When $\dim(M)= 0$, clearly $C = l(M)$ works.  Assume that the theorem holds for modules of dimension $< d$. Now let
$\dim(M) = d$. Let $R/P_1, \ldots, R/P_r$ be the (not necessarily distinct) factors appearing in
a prime cyclic filtration of $M$.  We note that, for each $0 \leq i \leq d$,
\[
\dfrac{l(H_i(x_1, \ldots, x_k;M))}{l(M/(x_1, \ldots, x_{k})M)} \leq \dfrac{\sum_{j=1}^r l(H_i(x_1, \ldots, x_{k};R/P_j))}{l(M/(x_1, \ldots, x_k)M)} \leq \sum_{j=1}^r \dfrac{ l(H_i(x_1, \ldots, x_k;R/P_j))}{l(M/((x_1, \ldots, x_k)+P_j)M)}.
\]
Now each $M/P_jM$ is a finitely generated faithful module over the
local domain $R/P_j$ with infinite residue field.  It then follows
from Lemma~\ref{lem:Length2} or Theorem~\ref{thm:Length2}
(applied to $M/P_jM$ and $R/P_j$) that
we may replace each term $l(M/((x_1, \ldots,
x_k)+P_j)M)$ by $l(R/((x_1, \ldots, x_k)+P_j)R)$ without affecting
the claim of the theorem.  We have now reduced to the
case of $M = R/P_j$ over the local domain $R/P_j$ of dimension $\leq d$.

Therefore, it suffices to verify the case of $M = R$ where $R$ is a
domain with $\dim(R) = d$. (We still have that
$R$ is a homomorphic image of a Gorenstein local ring $S$ with
infinite residue field.)

Let $(y_1, \dotsc, y_d)$ be a minimal reduction of $(x_1, \dotsc,
x_k)$. By Corollary \ref{cor.integralclosure} and Lemma~\ref{ses}, it suffices to
find a constant $D$ such that
\[
\frac{l(H_i(y_1, \ldots, y_{d};R))}
{l(R/(y_1,\ldots,y_{d}))} \le D
\]
for all systems of parameters $\uy:= y_1, \dotsc, y_d$ of $R$ and for all
$i = 0, \dotsc, d$. Now by Remark \ref{sseq}, it suffices to show that, for any
fixed finitely generated $R$-module $L$ with $\dim(L) < d$, there exists a
constant $D_L$ such that
\[
\frac{l(H_i(y_1, \ldots, y_{d};L))}
{l(R/(y_1,\ldots,y_{d}))} \le D_L
\]
independent of $\uy$ and $i$. Indeed, as $\l(L/(y_1,\ldots,y_{d})L)
\le \mu(L)l(R/(y_1,\ldots,y_{d}))$, we have
\[
\frac{l(H_i(y_1, \ldots, y_{d};L))}
{l(R/(y_1,\ldots,y_{d}))}
\le \mu(L) \frac{l(H_i(y_1, \ldots, y_{d};L))}
{l(L/(y_1,\ldots,y_{d})L)}.
\]
Since $\dim(L) < d$, the right hand side of the above inequality is
bounded above (independent of $\uy$ and $i$) by the inductive
hypothesis (noting that $2^d$ is a constant as well).
\end{proof}

It is well known that (for example, see \cite{Dut83}) if $R$ is a complete local
domain of characteristic $p>0$ and dimension $d \geq 1$, then, for every system of parameters $(x_1, \ldots, x_d)$ of
$R$, $\frac{l(H_i(x_1^{p^e}, \ldots, x_d^{p^e}; R))}{l(H_0(x_1^{p^e}, \ldots, x_d^{p^e}; R))}
\xrightarrow{e \rightarrow \infty} 0$ for each $1 \leq i \leq d$.
This classical result is essentially saying that the length of higher Koszul homology modules tends to $0$ compared to the length of the $0$-th Koszul homology module when we raise any system of parameters to high Frobenius powers. Our final result is a generalization of this result to a characteristic-free version.  More importantly, the convergence to $0$ occurs in a way that is independent of the system of parameters!

\begin{theorem}\label{Uniform0}
Let $R$ be a Noetherian local ring and $M$ be a finitely generated
quasi-unmixed $R$-module with $\dim(M) = d$.  For
every $\varepsilon>0$, there exists $t_0$ such that, for all $t \geq
t_0$, all systems of parameters $\ux:= x_1, \ldots, x_d$ of
$M$, and all $1 \leq i \leq d$,
\[
\frac{l(H_i(x_1^t, \ldots, x_d^t; M))}{l(M/(x_1^t, \ldots, x_d^t)M)}
<\varepsilon.
\]
In fact, there exists a constant $K$ such that for all $t \geq
1$, all systems of parameters $\ux:= x_1, \ldots, x_d$ of
$M$, and all $1 \leq i \leq d$,
\[
\frac{l(H_i(x_1^t, \ldots, x_d^t; M))}{l(M/(x_1^t, \ldots, x_d^t)M)}
\le \frac K{t^i}.
\]
\end{theorem}

\begin{proof}
As usual, we replace $R$ by $\overline{R}=R/\Ann(M)$ and complete $R$ to assume that $R$ is complete and that $M$ is faithful over $R$. Our hypothesis then implies that both $M$ and $R$ are equidimensional. (The rest of the proof only relies on the fact that $M$ is equidimensional and that $(R,\fm)$ is a homomorphic image of a Gorenstein local ring $S$.)

Because we consider only finitely many $i$, it is sufficient to fix
some $1 \leq i \leq d$.  By Remark \ref{sseq}, it suffices to show that, for
any fixed finitely generated $R$-module $L$ with $\dim(L) \le d-i$ and
any fixed $j = 0, \dotsc, d-i$, there
exists a constant $K$ such that, for all $t \ge 1$ and all $\ux$,
\[
\dfrac{l(H_j(x_1^t, \ldots, x_d^t; L))}{l(M/(x_1^t, \ldots, x_d^t)M)}
\le \frac K{t^i}.
\]
By taking a prime cyclic filtration of $L$, it suffices to show that,
for any fixed $P \in \Spec(R)$ such that $\dim(R/P) = d' \le d-i$,
there exists a constant $K$ such that, for all $t \ge 1$ and all $\ux$,
\[
\dfrac{l(H_j(x_1^t, \ldots, x_d^t;R/P))}{l(M/(x_1^t, \ldots, x_d^t)M)}
\le \frac K{t^i}.
\]
Denote $D := R/P$. By Theorem~\ref{Thm:HomoBound}, we fix
\[
C = \sup_{\substack{\sqrt{(x_1,\ldots, x_d)}=\fm}}
\left\{\frac{l(H_j(x_1, \ldots, x_d;D))}{l(D/(x_1,\ldots,x_d)D)}
\right\} <\infty.
\]
According to Theorem~\ref{Thm:Vogel}, we let
\[
B  = \sup_{\sqrt{I}=\fm}\left\{\frac{l(D/ID)}{e(I,D)} \right\}<\infty.
\]
Moreover, as $M$ is equidimensional, there exists $Q \in \Min(R) = \Min(M)$ such that $Q
\subsetneq P$ so that $D$ is a proper homomorphic image of $R/Q$, in which case
$\dim(R/Q) = d > d' = \dim(D)$ and $e(I, R/Q) \le e(I, M)$ for all
$\fm$-primary ideals $I$ (since $\dim(R/Q) = \dim(M)$). In light of Equation~\eqref{eqn:supR/P} in
the proof of Theorem~\ref{Thm:Vogel}, we set
\[
A = \sup_{\sqrt{I}=\fm}\left\{\frac{e(I, D)}{e(I, R/Q)}\right\}<\infty.
\]
Finally, we see
\begin{align*}
\frac{l(H_j(x_1^t, \dotsc, x_d^t;D))}{l(M/(x_1^t, \dotsc, x_d^t)M)}
&\leq \frac{l(H_j(x_1^t, \dotsc, x_d^t;D))}{e((x_1^t, \dotsc, x_d^t),M)} \\
&\leq C \frac{l(D/(x_1^t, \dotsc, x_d^t)D)}{e((x_1^t, \dotsc, x_d^t),M)} \\
&\leq BC \frac{e((x_1^t, \dotsc, x_d^t),D)}{e((x_1^t, \dotsc, x_d^t),M)} \\
& = BC \frac{t^{d'}e((x_1, \dotsc, x_d),D)}{t^de((x_1, \dotsc, x_d),M)} \\
& \leq \frac {BC}{t^{d-d'}}
\frac{e((x_1, \dotsc, x_d),D)}{e((x_1, \dotsc, x_d),R/Q)}
\leq \frac {ABC}{t^{d-d'}} \le \frac {ABC}{t^i},
\end{align*}
whose convergence to $0$, as $t \to \infty$, is independent of
systems of parameters $\ux:= x_1, \ldots, x_d$.
\end{proof}

\begin{lemma}\label{koszul-expo}
Let $R$ be a (Noetherian) ring, $M$ a finitely generated
$R$-module, $\ux:= x_1, \ldots, x_d$ a sequence of elements of $R$ such that
$l(M/(\ux)M) < \infty$, and $t_j \ge s_j \ge 1$ for $1 \le j \le
d$. Then for all $i$, we have
\[
l(H_i(x_1^{t_1}, \ldots, x_d^{t_d}; M))
\le l(H_i(x_1^{s_1}, \ldots, x_d^{s_d}; M)) \prod_{j=1}^d \frac{t_j}{s_j}
\]
\end{lemma}

\begin{proof}
By symmetry, it suffices to show
$l(H_i(x_1^{t}, x_2,\ldots, x_d; M))
\le l(H_i(x_1^{s}, x_2,\ldots, x_d; M)) \frac{t}{s}$ for all $t \ge s
\ge 1$. For each $i$, denote $H_i = H_i(x_2,\ldots, x_d; M)$. From the
exact sequence
\[
H_i \xrightarrow{x_1^t\cdot} H_i \to H_i(x_1^{t}, x_2,\ldots, x_d; M)
\to H_{i-1} \xrightarrow{x_1^t\cdot} H_{i-1}
\]
we see that
\begin{align*}
l(H_i(x_1^{t}, x_2,\ldots, x_d; M))
&= l(H_i/x_1^tH_i) + l((0:_{H_{i-1}} x_1^t)) \\
&=\sum_{j=1}^t l(x_1^{j-1}H_i/x_1^jH_i)
+ \sum_{j=1}^t l((0:_{H_{i-1}} x_1^j)/(0:_{H_{i-1}} x_1^{j-1})).
\end{align*}
Now, as we can do the above to $l(H_i(x_1^{s}, x_2,\ldots, x_d; M))$
as well,
it suffices to show the sequences $\{l(x_1^{j-1}H_i/x_1^jH_i)\}_j$ and
$\{l((0:_{H_{i-1}} x_1^j)/(0:_{H_{i-1}} x_1^{j-1}))\}_j$ are both
non-increasing. But this follows because the following maps induced by
multiplication by $x_1$:
\[
\frac{x_1^{j-1}H_i}{x_1^jH_i} \xrightarrow{x_1\cdot} \frac{x_1^{j}H_i}{x_1^{j+1}H_i}
\quad\text{and} \quad
\frac{(0:_{H_{i-1}} x_1^{j+1})}{(0:_{H_{i-1}} x_1^{j})} \xrightarrow{x_1\cdot}
\frac{(0:_{H_{i-1}} x_1^j)}{(0:_{H_{i-1}} x_1^{j-1})},
\]
are onto and 1-1 respectively.
\end{proof}

The next theorem generalizes the uniform convergence
established in Theorem~\ref{Uniform0}.

\begin{theorem}\label{Uniform0-ti}
Let $R$ be a Noetherian local ring and $M$ a finitely generated
quasi-unmixed $R$-module with $\dim(M) = d$.  For
every $\varepsilon>0$, there exists $t_0$ such that, for all $t_j \geq
t_0$ with $1 \le j \le d$, all systems of parameters $\ux:= x_1,
\ldots, x_d$ of $M$, and all $1 \leq i \leq d$,
\[
\frac{l(H_i(x_1^{t_1}, \ldots, x_d^{t_d}; M))}{l(M/(x_1^{t_1}, \ldots, x_d^{t_d})M)}
<\varepsilon.
\]
In fact, there exists a constant $K$ such that for all $t_j \geq
1$ with $1 \le j \le d$, all systems of parameters $\ux:= x_1, \ldots, x_d$ of
$M$, and all $1 \leq i \leq d$,
\[
\frac{l(H_i(x_1^{t_1}, \ldots, x_d^{t_d}; M))}{l(M/(x_1^{t_1}, \ldots, x_d^{t_d})M)}
\le \frac K{(\min_jt_j)^i}.
\]
\end{theorem}

\begin{proof}
With $t_j \geq 1$ for $1 \le j \le d$, we denote $t = \min_j
t_j$. Then for all systems of parameters $\ux:= x_1, \ldots, x_d$ of
$M$ and all $1 \leq i \leq d$, we have
\begin{align*}
\frac{l(H_i(x_1^{t_1}, \ldots, x_d^{t_d}; M))}{l(M/(x_1^{t_1}, \ldots, x_d^{t_d})M)}
&\le \frac{l(H_i(x_1^{t}, \ldots, x_d^{t}; M))}
{l(M/(x_1^{t_1}, \ldots, x_d^{t_d})M)}
\prod_{j= 1}^d\frac{t_j}t && \text{(Lemma~\ref{koszul-expo})}\\
&\le \frac{l(H_i(x_1^{t}, \ldots, x_d^{t}; M))}
{e((x_1^{t_1}, \ldots, x_d^{t_d}),M)} {d!e(\ol R)}
\prod_{j= 1}^d\frac{t_j}t  && \text{(Theorem~\ref{Thm:LechMod})}\\
&=\frac{l(H_i(x_1^{t}, \ldots, x_d^{t}; M))}
{e((x_1^{t}, \ldots, x_d^{t}),M)} {d!e(\ol R)}\\
&\le \frac{l(H_i(x_1^{t}, \ldots, x_d^{t}; M))}
{l(M/(x_1^{t}, \ldots, x_d^{t})M)} n(M) {d!e(\ol R)}
&& \text{(Theorem~\ref{Thm:Vogel})}
\end{align*}
in which $\ol R = R/\Ann(M)$.
Now Theorem~\ref{Uniform0} completes the proof.
\end{proof}

\begin{remark}\label{Uniform0SV}
We would like to mention that, in Theorem~\ref{Uniform0} (hence in
Theorem~\ref{Uniform0-ti}), the
assumption that $M$ is quasi-unmixed is necessary
{(at least when $R$ has infinite residue field)}.
In fact, the
conclusion of Theorem~\ref{Uniform0} (i.e., the uniform convergence to $0$) for
$M$ implies $n(M) < \infty$ {provided that $R$ has
  infinite residue field}, which forces $M$ to be quasi-unmixed by
\cite[Theorem 1]{SV96}.
For details, the existence of $t$ such that
$\frac{\sum_{i=1}^d(-1)^{i-1}l(H_i(x_1^t, \ldots, x_d^t; M))}
{l(M/(x_1^t, \ldots, x_d^t)M)}
<\varepsilon < 1$ for all systems of parameters $\ux:= x_1, \ldots,
x_d$ of $M$ implies
\begin{align*}
\dfrac{e((x_1, \ldots, x_d),M)}{l(M/(x_1, \ldots, x_d)M)}
& = \dfrac{e((x_1^t, \ldots, x_d^t), M)}{t^dl(M/(x_1, \ldots, x_d)M)}
\geq \dfrac{e((x_1^t, \ldots, x_d^t), M)}{t^dl(M/(x_1^t, \ldots, x_d^t)M)} \\
&\geq \dfrac{1}{t^d} \left(\dfrac{l(M/(x_1^t, \ldots,
    x_d^t)M)}{l(M/(x_1^t, \ldots, x_d^t)M)}
  - \dfrac{\sum_{i=1}^d(-1)^{i-1}l(H_i(x_1^t, \ldots, x_d^t; M))}
{l(M/(x_1^t, \ldots, x_d^t)M)}\right)
> \frac{1-\varepsilon}{t^d}
\end{align*}
for all systems of parameters $\ux:= x_1, \ldots, x_d$ of $M$, which
implies $n(M) < \infty$ {(given that $R$ has infinite
  residue field)}.
\end{remark}

\begin{remark}\label{SVbyInduction}
Evidently the results of this section rely on
Theorem~\ref{Thm:Vogel}. However, a careful analysis of the proofs in
this section reveals an alternative proof of
Theorem~\ref{Thm:Vogel} by induction on the dimension of the quasi-unmixed
$R$-module $M$ without explicit usage of homological degree or unmixed
degree. Without loss of generality, assume that $(R,\fm)$ is
complete with infinite residue field (thus, $R$ is a homomorphic image of a Gorenstein ring $S$ with
$\dim(S) = n$).
When $\dim(M) = 0$, it is clear that
Theorem~\ref{Uniform0},
Theorem~\ref{Thm:Vogel},
Lemma~\ref{lem:Length},
Corollary~\ref{cor.integralclosure},
Lemma~\ref{lem:Length2}, and
Theorem~\ref{Thm:HomoBound}
all hold. Now assume that \emph{all} these results hold
in dimension $<d$, and consider the case of dimension $d$. Then
Theorem~\ref{Uniform0} holds in dimension $d$ (because the proof of Theorem~\ref{Uniform0} only requires the aforementioned results in dimension $<d$), which implies
Theorem~\ref{Thm:Vogel} in dimension $d$ as we explained in
Remark~\ref{Uniform0SV}. Then we have
Lemma~\ref{lem:Length},
Corollary~\ref{cor.integralclosure},
Lemma~\ref{lem:Length2} and
Theorem~\ref{Thm:HomoBound} in dimension $d$, completing the
induction. Alternatively,
Theorem~\ref{Thm:HomoBound} and Theorem~\ref{Thm:Vogel} in dimension
$< d$ implies Theorem~\ref{Thm:Vogel} in dimension $d$ as follows:
It suffice to consider the case of $M = R$ being a domain. For an
arbitrary system of parameters $\uy:= y_1, \dotsc, y_d$ of $R$, we
have
\[
\frac{l(R/(\uy))}{e((\uy),R)}
\le \frac{e((\uy),R)+l(H_1(\uy;R))}{e((\uy),R)}
=1 + \frac{l(H_1(\uy;R))}{e((\uy),R)}.
\]
Similar to the reasoning in the proof of Theorem~\ref{Thm:HomoBound}
(plus taking prime cyclic filtration), it suffices to consider $R/P$
with $0 \neq P \in \Spec(R)$ and to find an upper bound for
$\frac{l(H_i(\uy;R/P))}{e((\uy),R)}$
for all systems of parameters $\uy$.
By Theorem~\ref{Thm:HomoBound} and Theorem~\ref{Thm:Vogel} in dimension
$< d$, it suffices to find an upper bound for
$
\frac{e(\uy,R/P)}{e((\uy),R)}
$
for all systems of parameters $\uy$. But this is Equation~\eqref{eqn:supR/P} in
the proof of Theorem~\ref{Thm:Vogel}.
\end{remark}

Even though the alternative proof sketched above does not use
extended degree explicitly, its approach is very similar to that of
homological degree: the alternative approach relies on Remark~\ref{sseq}
to reduce the dimension from $\dim(M)=d$ to $\dim(\Ext_S^{n-j}(M,S)) <
d$, $0 \le j < d$, while the homological degree involves
the same modules in its definition.

The following is an easy consequence of Theorem~\ref{Uniform0-ti}. It
says that for all systems of parameters $\ux = x_1, \dotsc, x_{d}$ on
$M$, the rate of convergence of
$\frac{l(M/(x_1^{t_1}, \dotsc,x_{d}^{t_{d}})M)}{\prod_{j=1}^{d}t_j}$
to $e((\ux),M)$ is uniformly controlled by $l(M/(\ux)M)$ and $\min_j t_j$ only.

\begin{corollary}\label{uniform-convergence}
Let $(R, \fm)$ be a Noetherian local ring and $M$ be a finitely generated quasi-unmixed $R$-module.
Then for every constant $C >0$ and every $\epsilon > 0$, there exists
$t_0 \in \mathbb N$ such that, for all $t_j \ge t_0$ with $1 \le j \le
d$, all systems of parameters $\ux = x_1, \dotsc, x_{d}$ on
$M$ such that $e((\ux),M) \le C$, we have
\[
0 \le \frac{l(M/(x_1^{t_1}, \dotsc,x_{d}^{t_{d}})M)}{\prod_{j=1}^{d}t_j} - e((\ux),M)
< \epsilon.
\]
In fact, there exists a constant $K$ such that,
for all $t_j \ge 1$ with $1 \le j \le d$ and all systems of parameters
$\ux = x_1, \dotsc, x_{d}$ on $M$, we have
\[
0 \le \frac{l(M/(x_1^{t_1}, \dotsc,x_{d}^{t_{d}})M)}{\prod_{j=1}^{d}t_j} - e((\ux),M)
\le e((\ux),M) \frac K{\min_j t_j}
\le l(M/(\ux)M) \frac K{\min_j t_j} .
\]
\end{corollary}

\begin{proof}
For all all systems of parameters $\ux = x_1, \dotsc, x_{d}$ on $M$,
we have (with $\ux^{[t]} := x_1^{t_1}, \dotsc, x_{d}^{t_{d}}$)
\begin{align*}
0  \le \frac{l(M/(\ux^{[t]})M)}{\prod_{j=1}^{d}t_j} - e((\ux),M)
&= \frac{l(M/(\ux^{[t]})M) - e((\ux^{[t]}),M)}{\prod_{j=1}^{d}t_j} \\
&= \frac{l(M/(\ux^{[t]})M)}{\prod_{j=1}^{d}t_j}
      \frac{\sum_{i=1}^{d}(-1)^{i-1} l(H_i(\ux^{[t]};M))}{l(M/(\ux^{[t]})M)}  \\
&\le\frac{e((\ux^{[t]}),M) n(M)}{\prod_{j=1}^{d}t_j}
      \frac{\sum_{i=1}^{d}(-1)^{i-1} l(H_i(\ux^{[t]};M))}{l(M/(\ux^{[t]})M)} \\
&= e((\ux),M) n(M) \frac{\sum_{i=1}^{d}(-1)^{i-1} l(H_i(\ux^{[t]};M))}{l(M/(\ux^{[t]})M)} \\
&\le e((\ux),M) \left(n(M) \sum_{i=1}^{d}(-1)^{i-1}\frac{
    l(H_i(\ux^{[t]};M))}{l(M/(\ux^{[t]})M)} \right).
\end{align*}
Now Theorem~\ref{Uniform0-ti} completes the proof.
\end{proof}

Finally, we remark that even though $n(M)$ could tend to $\infty$ as $M$ varies (see Example \ref{nlarge}), $\frac{n(M)}{\mu(M)}$ has an upper bound that depends only on $R/\Ann(M)$ (and not on $M$). We would like to thank the referee for pointing out this question.
\begin{remark}
\label{RemMinGen}
Let $(R,\fm)$ be a Noetherian local ring, and let $M$ be a finitely generated quasi-unmixed $R$-module. Set $\overline{R}=R/\Ann(M)$, which is quasi-unmixed by assumption and we have $\Assh(M)=\Assh(\overline{R})$. Let $c = \max_{P \in \Assh(\overline{R})} l_{\overline{R}_P}(\overline{R}_P)$. Clearly $l(M/IM)\leq \mu(M)l(\overline{R}/I\overline{R})$, and by the associativity formula for multiplicities $e(I, \overline{R})\leq ce(I, M)$.
Now for all $\fm$-primary ideals $I$ we have
$$\frac{l(M/IM)}{e(I,M)\mu(M)}= \frac{l(M/IM)}{\mu(M)l(\overline{R}/I\overline{R})}\cdot \frac{l(\overline{R}/I\overline{R})}{e(I, \overline{R})}\cdot \frac{e(I, \overline{R})}{e(I,M)}\leq cn(\overline{R}).$$ Therefore $\frac{n(M)}{\mu(M)}\leq cn(\overline{R})$, and the latter depends only on $\overline{R}=R/\Ann(M)$. Also note that if we take $R=k[[x,y]]$, $I_t=(x)\cap (x,y)^n$. Then $R/I_t$ is quasi-unmixed and $n(R/I_t)\geq t$: the ideal $(y)$ has multiplicity $1$ and colength $t$ in $R/I_t$. Therefore, in general the $n(\overline{R})$ (as $M$ varies) are not bounded in terms of invariants of $R$.
\end{remark}

\appendix
\section{Global version of the results}

In this appendix we briefly explain that our results and methods in
Section 4 work globally.
Most of the results rely on $\sup\{e(\ol R_P) \mid P \in
  \Spec(\ol R)\} < \infty$, with $\ol R = R/\Ann(M)$, and rely on the
  uniform Artin-Rees property (cf. \cite{Hun92}).
We observe that $\sup\{e(\ol R_P) \mid P \in \Spec(\ol R)\} <
\infty$ if the regular loci of all quotient domains of $\ol R$ are open
(e.g., $\ol R$ is excellent).
By \cite[Theorem~4.12]{Hun92}, the uniform Artin-Rees property holds for $R$
(hence holds for all its localizations $R_P$ with the same constant)
if $R$ is essentially of finite type over a Noetherian local ring or
$\mathbb Z$, or if $R$ is an F-finite Noetherian ring of prime
characteristic $p$.
We will also use the fact that a homomorphic image of a Cohen-Macaulay
local ring is quasi-unmixed if and only if it is equidimensional.

We start with the global version of 4.1--4.4.

\begin{lemma}
\label{lem:Length-g}
Let $R$ be a Noetherian ring that is a homomorphic image of a
Noetherian Gorenstein ring $S$ with $\dim (S)<\infty$, and let $M$ be
a finitely generated $R$-module such that $\ol R = R/\Ann(M)$ is
locally equidimensional and satisfies the uniform Artin-Rees property
with $\sup\{e(\ol R_P) \mid P \in \Spec(\ol R)\} < \infty$.
Suppose the residue field of $R_P$ is infinite for all $P\in\Supp(M)$. Then there exists a constant $C$ such that, for every $h \in \mathbb
R$ and every $P \in \Supp(M)$,
\[
\frac hC \le \inf_{\substack{\sqrt{I}= \sqrt{J} =P_P \\ e(I, M_P) \ge h e(J,M_P)}}
\left\{\frac{l(M_P/IM_P)}{l(M_P/JM_P)} \right\}
\quad \text{and} \quad
\sup_{\substack{\sqrt{I}= \sqrt{J} =P_P \\ e(I, M_P) \le h e(J,M_P)}}
\left\{\frac{l(M_P/IM_P)}{l(M_P/JM_P)} \right\} \le hC.
\]
By $\sqrt{I}= \sqrt{J} =P_P$, we regard $I$ and $J$ as
($P_P$-primary) ideals of $R_P$.
\end{lemma}

\begin{proof}
For each $P$, do the same proof for $M_P$ over $R_P$ as in Lemma \ref{lem:Length}, and note that $n(M_P)$, $\dim(M_P)$ and $e(\overline R_P)$ have global
upper bounds (see Remark \ref{GlobalSV}).
\end{proof}

\begin{corollary}
\label{cor.integralclosure-g}
With notation and assumptions as in Lemma \ref{lem:Length-g}, we have
$$\sup_{\substack{ P \in \Supp(M) \\ \sqrt{I}=P_P \\ I\subseteq J\subseteq \overline{I}}}
\left\{\frac{l(M_P/IM_P)}{l(M_P/JM_P)} \right\}<\infty.$$
\end{corollary}
\begin{proof}
Apply Lemma \ref{lem:Length-g} with $h=1$ to get the desired claim.
\end{proof}


\begin{lemma}\label{lem:Length2-g}
Let $R$ be a Noetherian ring that is a homomorphic image of a
Noetherian Gorenstein ring $S$ with $\dim (S)<\infty$, and let $N$ be
a finitely generated $R$-module such that $R/\Ann(N)$ is locally
equidimensional and satisfies the uniform Artin-Rees property
with $\sup\{e(\ol R_P) \mid P \in \Spec(\ol R)\} < \infty$.
Suppose the residue field of $R_P$ is infinite for all $P\in\Supp(N)$. Then there exists a constant $C_N>0$ depending only on $N$ such that
\[
\sup_{\sqrt{I} =P_P}\left\{\frac{l(N_P/IN_P)}{l(M_P/IM_P)} \right\} \le C_N
\]
for all $P \in \Supp(N)$, all $P_P$-primary ideals $I$, and all finitely
generated $R$-modules $M$ such that $\Supp(M_P) \supseteq \Supp(N_P)$.
(Note that such $M_P$ covers all finitely generated $R_P$-modules whose supports
contain the support of $N_P$.)
\end{lemma}
\begin{proof}
For each $P$, do the same proof for $N_P$ over $R_P$ as in Lemma \ref{lem:Length2}, and note that the constant ${n(N_P)c\dim(N_P)!e(\overline R_P)}$
has a global upper bound (see Remark \ref{GlobalSV}): we can take $c = \max_{P \in \Min(N)}
l_{R_P}(N_P)$, which is enough for each local consideration.
\end{proof}


\begin{theorem}\label{thm:Length2-g}
Let $R$ be a Noetherian ring that is a homomorphic image of a Noetherian Gorenstein ring $S$ with $\dim (S)<\infty$. Suppose $R$ is locally equidimensional and satisfies the uniform Artin-Rees property and that $\sup\{e(\ol R_P) \mid P \in \Spec(\ol R)\} < \infty$. Suppose the residue field of $R_P$ is infinite for all $P\in\Supp(R)$. Then, for all finitely generated $R$-modules $M$ and $N$, we have
\begin{enumerate}
\item $
\sup_{\sqrt{I} =P_P}\left\{\frac{l(N_P/IN_P)}{l(M_P/IM_P)}
\right\} < \infty
\text{ for all } P \in \Supp(N)
\iff \Supp(M) \supseteq \Supp(N)$.
\item There exists a constant $C>0$ depending only on $N$ such that
\[
1/C \le \inf_{\sqrt{I} =P_P}\left\{\frac{l(M_P/IM_P)}{l(N_P/IN_P)} \right\}
\quad \text{and, equivalently,} \quad
\sup_{\sqrt{I} =P_P}\left\{\frac{l(N_P/IN_P)}{l(M_P/IM_P)} \right\} \le C
\]
for all $P \in \Supp(N)$ and
all (finitely generated $R$-modules) $M$ with $\Supp(M_P) \supseteq \Supp(N_P)$.
\item $0 < \inf_{\sqrt{I} =P_P}\left\{\frac{l(M_P/IM_P)}{l(N_P/IN_P)} \right\}
\le \sup_{\sqrt{I} =P_P}\left\{ \frac{l(M_P/IM_P)}{l(N_P/IN_P)} \right\} < \infty
\iff \Supp(M) = \Supp(N)$.
\end{enumerate}
\end{theorem}
\begin{proof}
For each $P$, do the same proof as in Theorem \ref{thm:Length2}, replacing Lemma \ref{lem:Length2} by Lemma \ref{lem:Length2-g}.
\end{proof}

The next lemma is the same as the local version; we record it here only to preserve the numbering.

\begin{lemma}
\label{ses-g}
Let $(R,\fm)$ be a Noetherian local ring and $M$ a finitely generated $R$-module.
If $(y_1, \ldots, y_d) \subseteq (x_1, \ldots, x_k)$ are $\fm$-primary
ideals of $R$, then for all $0 \leq i \leq k$,
\[
l(H_i(x_1, \ldots, x_k;M))
\leq \sum_{j=0}^k {k \choose j} l(H_{i-j}(y_1, \ldots, y_d;M))
\le 2^k \max_{0 \le j \le k}  l(H_{i-j}(y_1, \ldots, y_d;M)),
\]
with the convention that $H_{<0}(y_1, \ldots, y_d; M)=0$.
\end{lemma}


The next remark follows immediately from the local version.

\begin{remark}[\cite{Schenzel98} or \cite{BHM18}]
\label{sseq-g}
Let $R$ be a Noetherian ring that is a homomorphic image of a Noetherian Gorenstein ring $S$. Then, for every
finitely generated $R$-module $M$, every $P \in \Spec(R)$, every
system of parameters $\underline{x} = x_1, \dotsc, x_{\dim(R_P)}$ of
$R_P$, and every $i = 1, \dotsc, \dim(R_P)$, we have
\[
l(H_i(x_1, \ldots, x_{\dim(R_P)};M_P)) \leq
\sum_{j = 0}^{\dim(R_P)-i} l(H_{\dim(R_P)-i-j}(x_1, \ldots, x_{\dim(R_P)};\Ext_{S_P}^{\dim(S_P)-j}(M_P,S_P))).
\]
Note that $\dim((\Ext_{S}^{\dim(S_P)-j}(M,S))_P) \le \dim(R_P)-i$ for
each $j = 0, \dotsc, \dim(R_P)-i$. Here, by abuse of notation, $S_P$
stands for $S_Q$ where $Q$ is the pre-image of $P$ under the onto map $S \to R$.
\end{remark}

\begin{theorem} \label{Thm:HomoBound-g}
Let notation and assumptions be as in Lemma \ref{lem:Length-g}.
Then there exists $C$ depending on $M$ such
that, for every $P \in \Supp(M)$ and every $k \geq \dim(M_P)$, we have
\[
\sup_{\substack{\sqrt{(x_1,\ldots, x_{k}) + \Ann(M_P)}=P_P\\ 0 \leq i \leq k}}
\left\{\frac{l(H_i(x_1, \ldots, x_{k};M_P))}
{l(M_P/(x_1,\ldots,x_{k})M_P)} \right\} \le 2^kC.
\]
By $\sqrt{(x_1,\ldots, x_{k}) + \Ann(M)_P}=P_P$, the understanding is
that $x_i \in P_P$.
\end{theorem}

\begin{proof}
We use induction on $d$ to prove the following claim: \emph{For every
  $d \ge 0$, there exists $C_{M,d}$ depending on $M$ and $d$ such
that, for every $P \in \Supp(M)$ with $\dim(M_P) = d$, every $k
\geq d$, we have
\[
\sup_{\substack{\sqrt{(x_1,\ldots, x_{k}) + \Ann(M_P)}=P_P\\ 0 \leq i \leq k}}
\left\{\frac{l(H_i(x_1, \ldots, x_{k};M_P))}
{l(M_P/(x_1,\ldots,x_{k})M_P)} \right\} \le 2^kC_{M,d}.
\]
}%
As $\dim(M) < \infty$, we can take $C = \max_{0 \le i \le
  \dim(M)}C_{i}$ to complete the proof of the theorem. We can replace $R$ by $R/\Ann(M)$, so $M$ is a faithful $R$-module. In
this sense, we are doing induction on the height of $P$.

When $d = 0$, $C_{M,0} = \max_{P \in \Min(M)} l_{R_P}(M_P)$ works.
Now let $d > 0$, and assume that the claim holds for $\le d-1$. 
We fix a prime cyclic filtration of $M$ over $R$ (which is independent
of the choice of $P$ at which we are localizing).
Localizing this prime cyclic filtration of $M$ at $P$, applying a similar argument as in the proof of
Theorem~\ref{Thm:HomoBound} in light of Lemma~\ref{lem:Length2-g} or
Theorem~\ref{thm:Length2-g}, we may further assume that $M=R$ and that
$R$ is a domain.

Now we prove the claim for $d$. By
Corollary~\ref{cor.integralclosure-g} and Lemma~\ref{ses-g}, it suffices
to find a constant $D$ such that
\[
\frac{l(H_i(y_1, \ldots, y_{d};R_P))}
{l(R_P/(y_1,\ldots,y_{d}))} \le D
\]
for all $P \in \Spec(R)$ such that $\dim(R_P) = d$, all systems of
parameters $\uy:= y_1, \dotsc, y_d$ of $R_P$, and all
$i = 0, \dotsc, d$. Now by Remark \ref{sseq-g}, it suffices to show that, for any
fixed finitely generated $R$-module $L$ satisfying $\dim(L_P) <
d$ whenever $\dim(R_P) = d$ (think of $L$ as one of those (finitely
many) global $\Ext$ modules over $R$), there exists a
constant $D_L$ such that
\[
\frac{l(H_i(y_1, \ldots, y_{d};L_P))}
{l(R_P/(y_1,\ldots,y_{d}))} \le D_L
\]
independent of $\uy$, $i$, and $P \in \Spec(R)$ such that $\dim(R_P) = d$. Indeed, as $\l(L_P/(y_1,\ldots,y_{d})L_P) \le \mu_R(L)l(R_P/(y_1,\ldots,y_{d}))$, we have
\[
\frac{l(H_i(y_1, \ldots, y_{d};L_P))}
{l(R_P/(y_1,\ldots,y_{d}))}
\le \mu(L) \frac{l(H_i(y_1, \ldots, y_{d};L_P))}
{l(L_P/(y_1,\ldots,y_{d})L_P)}.
\]
Since $\dim(L_P) < d$ when $\dim(R_P) = d$, the right hand side of the
above inequality is bounded above (independent of $P \in \Spec(R)$
such that $\dim(R_P) = d$, $\uy$ and $i$) by the inductive
hypothesis (noting that $2^d$ is a constant as well).
\end{proof}


\begin{theorem}\label{Uniform0-g}
Let notation and assumptions be as in Lemma \ref{lem:Length-g}. Then for every $\varepsilon>0$, there exists $t_0$ such that, for all $t \geq
t_0$, all $P \in \Supp(M)$, all systems of parameters $\ux:= x_1, \ldots, x_{\dim(M_P)}$ of
$M_P$, and all $1 \leq i \leq \dim(M_P)$,
\[
\frac{l(H_i(x_1^t, \ldots, x_{\dim(M_P)}^t; M_P))}{l(M_P/(x_1^t, \ldots, x_{\dim(M_P)}^t)M_P)}
<\varepsilon.
\]
In fact, there exists a constant $K$ such that for all $t \geq 1$,
all $P \in \Supp(M)$, all systems of parameters $\ux:= x_1, \ldots, x_{\dim(M_P)}$ of
$M_P$, and all $1 \leq i \leq \dim(M_P)$,
\[
\frac{l(H_i(x_1^t, \ldots, x_{\dim(M_P)}^t; M_P))}{l(M_P/(x_1^t, \ldots, x_{\dim(M_P)}^t)M_P)}
\le \frac K{t^i}.
\]
\end{theorem}

\begin{proof}
As usual, we replace $R$ by $\overline{R}=R/\Ann(M)$ to assume that
$M$ is faithful over $R$. Our hypothesis then implies that both $M$
and $R$ are locally equidimensional. 

Denote $\mathcal L = \{\Ext_S^j(M,S) \mid 0 \le j \le \dim(S)\}$, which
is a finite set. For each $L \in \mathcal L$, fix a prime cyclic
filtration $F_L$ of $L$, and denote
$\mathcal D = \{\text{all the $R/\fa$ appearing as a factor of $F_L$
  for some $L \in \mathcal L$}\}$, which is finite.

Fix $\varepsilon > 0$. Because we consider only finitely many $i$, it
is sufficient to fix some $1 \leq i \leq \dim(R)$.
By Remark \ref{sseq-g}, it suffices to show that there
exists a constant $K$ such that, for all $P \in \Supp(M)$, for all $L \in
\mathcal L$ satisfying $\dim(L_P) \le \dim(R_P)-i$, for all $t \geq 1$,
for all systems of parameters $\ux$ of $R_P$,
and for all $j = 0, \dotsc, \dim(R_P)-1$,
\[
\dfrac{l(H_j(x_1^t, \ldots, x_{\dim(R_P)}^t; L_P))}{l(M_P/(x_1^t, \ldots, x_{\dim(R_P)}^t)M_P)}
\le \frac K{t^i}.
\]
Then, via $F_L$, it suffices to show that there
exists a constant $K$ such that, for all $P \in \Supp(M)$, for all $D \in
\mathcal L$ satisfying $\dim(D_P) \le \dim(R_P)-i$, for all $t \geq 1$,
for all systems of parameters $\ux$ of $R_P$,
and for all $j = 0, \dotsc, \dim(R_P)-1$,
\[
\dfrac{l(H_j(x_1^t, \ldots, x_{\dim(R_P)}^t; D_P))}{l(M_P/(x_1^t, \ldots, x_{\dim(R_P)}^t)M_P)}
\le \frac K{t^i}.
\tag{$\dagger$} \label{dagger}
\]
By Theorem~\ref{Thm:HomoBound-g} and noting that $\mathcal D$ is
finite, we fix
\[
C = \sup_{\substack{D \in \mathcal D, \, P \in \Supp(D) \\ \sqrt{(x_1,\ldots, x_d)}=P_P}}
\left\{\frac{l(H_j(x_1, \ldots, x_d;D_P))}{l(D_P/(x_1,\ldots,x_d)D_P)}
\right\} <\infty.
\]
According to Remark \ref{GlobalSV} (i.e., the global version of Theorem~\ref{Thm:Vogel}), we let
\[
B  = \sup_{\substack{D \in \mathcal D, \, P \in \Supp(D)\\ \sqrt{I}=P_P}}\left\{\frac{l(D_P/ID_P)}{e(I,D_P)} \right\}<\infty.
\]
Moreover, for every $D = R/\fa \in \mathcal D$, fix $Q(D) \in \Min(R) =
\Min(M)$ such that $Q(D) \subseteq \fa$, so that $D = R/\fa$ is a
homomorphic image of $R/Q(D)$. Note that $e(I, R_P/Q(D)_P) \le e(I,
M_P)$ for all $P_P$-primary ideals $I$ of $R_P$ because $R$ is locally equidimensional (remember we already reduced to the case $R=R/\Ann(M)$). In light of the global version of Equation~\eqref{eqn:supR/P} (since we assume that $R/\Ann(M)$ satisfies the uniform Artin-Rees property, this equation holds by the same argument as in Lemma \ref{lem:Huneke}), we set
\[
A = \sup_{\substack{D \in \mathcal D, \, P \in \Supp(D)\\ \sqrt{I}=P_P}}\left\{\frac{e(I, D_P)}{e(I, (R/Q(D))_P)}\right\}<\infty.
\]

At this point we see that, for all $P \in \Supp(M)$, for all $D \in
\mathcal L$ satisfying $\dim(D_P)=d' \le d - i= \dim(R_P)-i =
\dim((R/Q(D))_P)-i$, for all $t \ge 1$, for all systems of parameters
$\ux$ of $R_P$, and for all $j = 0, \dotsc, \dim(R_P)-1$,
we have (denoting $d:= \dim((R/Q(D))_P)$)
\begin{align*}
\frac{l(H_j(x_1^t, \dotsc, x_d^t;D_P))}{l(M/(x_1^t, \dotsc, x_d^t)M_P)}
&\leq \frac{l(H_j(x_1^t, \dotsc, x_d^t;D_P))}{e((x_1^t, \dotsc, x_d^t),M_P)} \\
&\leq C \frac{l(D_P/(x_1^t, \dotsc, x_d^t)D_P)}{e((x_1^t, \dotsc, x_d^t),M_P)} \\
&\leq BC \frac{e((x_1^t, \dotsc, x_d^t),D_P)}{e((x_1^t, \dotsc, x_d^t),M_P)} \\
& = BC \frac{t^{d'}e((x_1, \dotsc, x_d),D_P)}{t^de((x_1, \dotsc, x_d),M_P)} \\
& \leq \frac {BC}{t^{d-d'}}
\frac{e((x_1, \dotsc, x_d),D_P)}{e((x_1, \dotsc, x_d),(R/Q(D))_P)}
\leq \frac {ABC}{t^{d-d'}} \le \frac {ABC}{t^i} ,
\end{align*}
whose convergence to $0$, as $t \to \infty$, gives what we need in
\eqref{dagger} to complete the proof.
\end{proof}


Similarly, we can generalize Theorem~\ref{Uniform0-g} as follows:

\begin{theorem}\label{Uniform0-ti-g}
Let notation and assumptions be as in Lemma \ref{lem:Length-g}. Then
there exists a constant $K$ such that, for all $P
\in \Supp(M)$, all $t_j \geq 1$ with $1 \le j \le \dim(M_P)$, all
systems of parameters $\ux:= x_1, \ldots, x_{\dim(M_P)}$ of
$M_P$, and all $1 \leq i \leq \dim(M_P)$,
\[
\frac{l(H_i(x_1^{t_1}, \ldots, x_{\dim(M_P)}^{t_{\dim(M_P)}}; M_P))}
{l(M_P/(x_1^{t_1}, \ldots, x_{\dim(M_P)}^{t_{\dim(M_P)}})M_P)}
\le \frac K{(\min_j t_j)^i}.
\]
\end{theorem}

\begin{proof}
We carry out the same proof as in Theorem~\ref{Uniform0-ti} for each $M_P$, using the
fact that $\dim(M_P)$, $e(\ol R_P)$, and $n(M_P)$ have global upper bounds for
all $P \in \Supp(M)$ (see Remark~\ref{GlobalSV} or Remark~\ref{SVbyInduction-g}).
\end{proof}

\begin{remark}\label{Uniform0SV-g}
We would like to mention that, in Theorem~\ref{Uniform0-g}, the
assumption that $M$ is locally equidimensional is necessary.
Localized at each $P$, it is the same argument as
Remark~\ref{Uniform0SV}.
Moreover, the conclusion of Theorem~\ref{Uniform0-g} implies the conclusion of Remark \ref{GlobalSV} (i.e., the global version of Theorem \ref{Thm:Vogel}) following the same argument as in Remark~\ref{Uniform0SV} applied to $M_P$ for all $P$.
\end{remark}

\begin{remark}\label{SVbyInduction-g}
Evidently the results of this section rely on the global version of Theorem \ref{Thm:Vogel} (i.e., Remark \ref{GlobalSV}). However, a careful analysis of the proofs in this section reveals an alternative proof of the global version of
Theorem~\ref{Thm:Vogel} by induction on dimension of the $R$-module $M$. We use notation as in Lemma \ref{lem:Length-g}. When $\dim(M) = 0$, it is clear that
Theorem~\ref{Uniform0-g},
Remark~\ref{GlobalSV},
Lemma~\ref{lem:Length-g},
Corollary~\ref{cor.integralclosure-g},
Lemma~\ref{lem:Length2-g} and
Theorem~\ref{Thm:HomoBound-g}
all hold. Now assume that \emph{all} these results hold
in dimension $<d$; and consider the case of dimension $d$. Then
Theorem~\ref{Uniform0-g} holds in dimension $d$ (because the proof of
Theorem~\ref{Uniform0-g} only requires the aforementioned results in
dimension $<d$), which implies
Remark~\ref{GlobalSV} in dimension $d$ as we explained in
Remark~\ref{Uniform0SV-g}. Then we have
Lemma~\ref{lem:Length-g},
Corollary~\ref{cor.integralclosure-g},
Lemma~\ref{lem:Length2-g} and
Theorem~\ref{Thm:HomoBound-g} in dimension $d$, completing the
induction.
\end{remark}

The following is an easy consequence of Theorem~\ref{Uniform0-ti-g}.

\begin{corollary}\label{uniform-convergence-g}
Let notation and assumptions be as in Lemma \ref{lem:Length-g}.
Then, for every constant $C >0$ and every $\epsilon > 0$, there exists
$t_0 \in \mathbb N$ such
that, for all $P \in \Supp(M)$, all $t_j \ge t_0$ with $1 \le j \le
\dim(M_P)$, and all systems of parameters $\ux = x_1, \dotsc, x_{\dim(M_P)}$ on
$M_P$ such that $e((\ux),M_P) \le C$, we have
\[
0 \le \frac{l(M_P/(x_1^{t_1}, \dotsc,x_{\dim(M_P)}^{t_{\dim(M_P)}})M_P)}{\prod_{j=1}^{\dim(M_P)}t_j} - e((\ux),M_P)
< \epsilon.
\]
In fact, there exists a constant $K$ such that
for all $P \in \Supp(M)$, all $t_j \ge 1$ with $1 \le j \le
\dim(M_P)$, and all systems of parameters $\ux = x_1, \dotsc, x_{\dim(M_P)}$ on
$M_P$, we have
\[
0 \le \frac{l(M_P/(x_1^{t_1}, \dotsc,x_{\dim(M_P)}^{t_{\dim(M_P)}})M_P)}{\prod_{j=1}^{\dim(M_P)}t_j} - e((\ux),M_P)
\le e((\ux),M_P) \frac K{\min_j t_j}
\le l(M_P/(\ux)M_P) \frac K{\min_j t_j} .
\]
\end{corollary}
\begin{proof}
We carry out the same proof as in Corollary~\ref{uniform-convergence} for each $M_P$, replacing Theorem~\ref{Uniform0-ti} by Theorem~\ref{Uniform0-ti-g}.
\end{proof}

For example, (the first part of) Corollary~\ref{uniform-convergence-g} applies to all
minimal reductions of $P_P/\Ann(M_P)$ in $R_P/\Ann(M_P)$
for all $P \in \Supp(M)$, since there is an upper bound for $\{e(M_P)
\mid P \in \Supp(M)\}$ under the assumption of
Corollary~\ref{uniform-convergence-g}.

Finally, similar to Remark~\ref{RemMinGen}, we have the following

\begin{remark}
\label{RemMinGen-g}
Let notation and assumptions be as in Lemma \ref{lem:Length-g} and 
let $c = \max_{P \in \Min(\overline{R})} l_{\overline{R}_P}(\overline{R}_P)$. 
By Remark~\ref{RemMinGen} and Remark~\ref{GlobalSV} (or
Remark~\ref{SVbyInduction-g}), we have  
\[
\frac{\sup_{P \in \Supp(M)}n(M_P)}{\mu(M)}
\leq \sup_{P \in \Supp(M)}\frac{n(M_P)}{\mu(M_P)}
\leq c \sup_{P \in \Supp(M)}n(\overline R_P) < \infty,
\]
in which $c\sup_{P \in \Supp(M)}n(\overline R_P)$ depends only
on $\overline R = R/\Ann(M)$.
\end{remark}


\begin{thebibliography}{99}

\bibitem{AH00} N. Allsop and L. T. Hoa. \emph{On the quotient between length and multiplicity}, Comm. Algebra, 28(2):815--828, 2000.



\bibitem{BHM18} B. Bhatt and M. Hochster and L. Ma. \emph{Lim Cohen-Macaulay sequences of modules}, in preparation.

\bibitem{Bou06} N. Bourbaki, \emph{\'{E}l\'{e}ments de mathematique. Alg\`{e}bre commutative. Chapitres 8 et 9}, Springer, Berlin, 2006, Reprint of the 1983 original.

\bibitem{CDHZ12}
O. Celikbas, H. Dao, C. Huneke, and Y, Zhang, \emph{Bounds on the Hilbert-Kunz multiplicity}, Nagoya Math. J. 205 (2012), 149--165.

\bibitem{CQ17}
N. T. Cuong and P. H. Quy, \emph{On the structure of finitely generated modules over quotients of Cohen-Macaulay local rings}, arXiv: 1612.07638.




\bibitem{Dut83} S.P. Dutta. \emph{Frobenius and multiplicities}, J. Algebra., 85(2):424--448, 1983.

\bibitem{Han02} D. Hanes.  \emph{Bounds on multiplicities of local rings}, Comm. Algebra, 30(8):3789-3812, 2002.



\bibitem{Hun92}
C. Huneke, \emph{Uniform bounds in noetherian rings}, Invent. Math. 107 (1992), 203--223.

\bibitem{HSV17}
C. Huneke, Ilya Smirnov and J. Validashti, \emph{A generalization of an inequality of Lech relating multiplicity and colength}, arXiv:1711.06951.

\bibitem{HS06}
C. Huneke and I. Swanson, \emph{Integral closure of ideals, rings and modules}, London Mathematical Society Lecture Note Series, 336. Cambridge University Press, Cambridge, 2006.

\bibitem{Lech57}
C. Lech,  \emph{On the associativity formula for multiplicities}, Ark. Mat. 3 (1957), 301--314.

\bibitem{Lech60}
C. Lech,  \emph{Note on multiplicities of ideals}, Ark. Mat. 4 (1960), 63--86.

\bibitem{Ma17}
L. Ma,  \emph{Lech's conjecture in dimension three}, Adv. Math. 322 (2017), 940--970.

\bibitem{Mum77}
D. Mumford,  \emph{Stability of projective varieties}, L'Enseignement Math\'{e}matique, Vol 23 (1977).

\bibitem{Matsumura86}
H. Matsumura, \emph{Commutative ring theory}, Cambridge Sudies in Advanced Mathematics, 8. Cambridge University Press, Cambridge, 1986, xiv+320 pp.

\bibitem{MV95}
C. Miyazaki and W. Vogel, \emph{A new invariant of local rings}, C.R. Math. Rep. Acad. Sci. Canada 17 (1995), 117--122.


\bibitem{Nagata56}
M. Nagata,  \emph{The theory of multiplicity in general local rings}, Proceedings of the International Symposium on algebraic number theory, Tokyo \& Nikko, 1955, pp. 191--226.


\bibitem{Schenzel83}
P. Schenzel,  \emph{Standard systems of parameters and their blowing-up rings}, J. Reine Angew. Math. 344 (1983), 201--220.

\bibitem{Schenzel98}
P. Schenzel,  \emph{On the use of local cohomology in algebra and geometry}, Six lectures on commutative algebra, 241--292, Progr. Math., 166, Birkh\"{a}user, Basel, 1998.


\bibitem{SV78}
J. St\"{u}ckrad and W. Vogel,  \emph{Toward a theory of Buchsbaum singularities}, Amer. J. Math 100 (1978), no. 4., 727--746.

\bibitem{SV96}
J. St\"{u}ckrad and W. Vogel,  \emph{On composition series and new invariants of local algebra}, preprint MPI 96--99, Bonn 1996.



\bibitem{Vas98}
W. V. Vasconcelos, \emph{The homological degree of a module}, Tran. Amer. Math. Soc. 350 (1998), 1167--1179.

\bibitem{VasSixLec98}
W. V. Vasconcelos, \emph{Cohomological degrees of graded modules}, In: Elias, J. (ed.) et al., Six lectures on commutative algebra. Basel (1998), 345--392.

\bibitem{Wat03}
K.-i. Watanabe, \emph{Chains of integrally closed ideals}, Comtemp. Math., 331, Amer. Math. Soc., Providence, RI, 2003, 353--358.



\end{thebibliography}
\end{document}